\documentclass{amsart}

\usepackage{ae,aecompl}
\usepackage{chngcntr}
\usepackage{graphicx}
\usepackage[all,cmtip]{xy}
\usepackage{amsmath, amscd}
\usepackage{amsthm}
\usepackage{amssymb}
\usepackage{amsfonts}
\usepackage{qsymbols}
\usepackage{latexsym}
\usepackage{mathrsfs}
\usepackage{cite}
\usepackage{color}
\usepackage{url}
\usepackage{enumerate}

\usepackage{verbatim}

\usepackage{hyperref}

\allowdisplaybreaks

\newcommand {\abs}[1]{\lvert#1\rvert}

\newcommand {\Be}{{B}}

\newcommand {\C}{{\mathbb C}}
\newcommand {\ce}{\mathrm{c}}

\newcommand {\ud}{\mathrm{d}}
\newcommand {\ue}{\mathrm{e}}

\newcommand {\Ell}{L}
\newcommand {\Ellp}{L^{p}}
\newcommand {\Ellpprime}{L^{p'}}
\newcommand {\Ellq}{L^{q}}
\newcommand {\Ellqprime}{L^{q'}}
\newcommand {\Ellr}{L^{r}}
\newcommand {\Ellone}{L^{1}}
\newcommand {\Elltwo}{L^{2}}
\newcommand {\Ellinfty}{L^{\infty}}

\newcommand {\F}{{\mathcal{F}}}
\newcommand {\Hr}{H}

\newcommand {\ui}{\mathrm{i}}
\newcommand {\I}{{\mathrm{I}}}

\newcommand {\ind}{{\mathbf{1}}}

\newcommand{\lb}{\langle}
\newcommand{\rb}{\rangle}
\newcommand {\La}{{\mathcal{L}}}
\newcommand {\calL}{{\mathcal{L}}}

\newcommand {\N}{{{\mathbb N}}}

\newcommand {\norm}[1]{\left\|#1\right\|}
\newcommand{\one}{{{\bf 1}}}
\newcommand {\ph}{{\varphi}}
\newcommand {\Pa}{{\mathcal{P}}}
\newcommand {\R}{{\mathbb R}}
\newcommand {\Rd}{{\mathbb{R}^{d}}}
\newcommand {\Rad}{{\mathrm{Rad}}}

\newcommand {\supp}{{\mathrm{supp}}}

\newcommand {\Sw}{\mathcal{S}}
\newcommand {\Cs}{{\mathscr C}}
\newcommand {\schatten}{\Cs}
\newcommand {\T}{{\mathbb T}}
\newcommand {\w}{{\omega}}

\newcommand {\Z}{{{\mathbb Z}}}

\newcommand {\vanish}[1]{\relax}

\newcommand{\wh}{\widehat}

\renewcommand{\restriction}{\mathord{\upharpoonright}}

\newtheorem{theorem}{Theorem}[section]
\newtheorem{lemma}[theorem]{Lemma}
\newtheorem{proposition}[theorem]{Proposition}
\newtheorem{corollary}[theorem]{Corollary}
\newtheorem{problem}[theorem]{Problem}

\theoremstyle{definition}

\newtheorem{remark}[theorem]{Remark}
\newtheorem{example}[theorem]{Example}

\numberwithin{equation}{section}

\title[Fourier multiplier theorems involving type and cotype]{Fourier multiplier theorems involving type and cotype}

\author{Jan Rozendaal}
\address{Institute of Mathematics Polish Academy of Sciences\\
ul.~\'{S}niadeckich 8\\
00-656 Warsaw\\
Poland}
\email{janrozendaalmath@gmail.com}

\author{Mark Veraar}
\address{Delft Institute of Applied
Mathematics\\
Delft University of Technology\\
P.O.~Box 5031\\
2628 CD Delft\\
The Netherlands}
\email{M.C.Veraar@tudelft.nl}



\thanks{The second author is supported by the VIDI subsidy 639.032.427 of the Netherlands Organisation for Scientific Research (NWO)}

\begin{document}

\begin{abstract}
In this paper we develop the theory of Fourier multiplier operators $T_{m}:L^{p}(\R^{d};X)\to L^{q}(\R^{d};Y)$, for Banach spaces $X$ and $Y$, $1\leq p\leq q\leq \infty$ and $m:\R^d\to \mathcal{L}(X,Y)$ an operator-valued symbol. The case $p=q$ has been studied extensively since the 1980's, but far less is known for $p<q$. In the scalar setting one can deduce results for $p<q$ from the case $p=q$. However, in the vector-valued setting this leads to restrictions both on the smoothness of the multiplier and on the class of Banach spaces. For example, one often needs that $X$ and $Y$ are UMD spaces and that $m$ satisfies a smoothness condition.

We show that for $p<q$ other geometric conditions on $X$ and $Y$, such as the notions of type and cotype, can be used to study Fourier multipliers. Moreover, we obtain boundedness results for $T_m$ without any smoothness properties of $m$. Under smoothness conditions the boundedness results can be extrapolated to other values of $p$ and $q$ as long as $\tfrac{1}{p}-\tfrac{1}{q}$ remains constant.
\end{abstract}

\maketitle

\section{Introduction}

Fourier multiplier operators play a major role in analysis and in particular in the theory of partial differential equations. Such operators are of the form
\begin{align*}
T_{m}(f)  = \F^{-1} (m \F f),
\end{align*}
where $\F$ denotes the Fourier transform and $m$ is a function on $\Rd$. Usually one is interested in the boundedness of  $T_m:\Ellp(\Rd)\to \Ellq(\Rd)$ with $1\leq p\leq q\leq \infty$ (the case $p>q$ is trivial by \cite[Theorem 1.1]{Hormander60}). The class of Fourier multiplier operators coincides with the class of singular integral operators of convolution type $f\mapsto K*f$, where $K$ is a tempered distribution.

The simplest class of examples of Fourier multipliers can be obtained by taking $p=q=2$. Then $T_m$ is bounded if and only if $m\in \Ellinfty(\Rd)$, and $\|T_m\|_{\La(\Elltwo(\Rd))} = \|m\|_{\Ellinfty(\Rd)}$. For $p=q=1$ and $p=q=\infty$ one obtains only trivial multipliers, namely Fourier transforms of bounded measures. The case where $p=q\in (1, \infty)\setminus\{2\}$ is highly nontrivial. In general only sufficient conditions on $m$ are known that guarantee that $T_m$ is bounded, although also here it is necessary that $m\in \Ellinfty(\Rd)$.

In the classical paper \cite{Hormander60} H\"ormander studied Fourier multipliers and singular integral operators of convolution type. In particular, he showed that if $1<p\leq 2\leq q<\infty$, then
\begin{align}\label{weakLrcondition}
T_m:\Ellp(\Rd)\to \Ellq(\Rd) \text{ is bounded if } m\in \Ell^{r,\infty}(\Rd)\text{ with }\tfrac{1}{r} = \tfrac{1}{p}-\tfrac{1}{q}.
\end{align}
Here $\Ell^{r,\infty}(\Rd)$ denotes the weak $L^r$-space.
In particular, every $m$ with $|m(\xi)|\leq C|\xi|^{-d/r}$ satisfies $m\in \Ell^{r,\infty}(\Rd)$.
It was also shown that the condition $p\leq 2\leq q$ is necessary here. More precisely, if there exists a function $F$ such that $\{F>0\}$ has nonzero measure and for all $m:\R^d\to \R$ with $|m|\leq |F|$, $T_m:\Ellp(\Rd)\to \Ellq(\Rd) $ is bounded, then $p\leq 2\leq q$.

H\"ormander also introduced an integral/smoothness condition on the kernel $K$ which allows one to extrapolate the boundedness of $T_{m}$ from $\Ell^{p_0}(\Rd)$ to $\Ell^{q_0}(\Rd)$ for some $1<p_{0}\leq q_{0}<\infty$ to boundedness of $T_{m}$ from $\Ellp(\Rd)$ to $\Ellq(\Rd)$ for all $1<p\leq q<\infty$ satisfying $\tfrac{1}{p} - \tfrac{1}{q} = \tfrac{1}{p_0} - \tfrac{1}{q_0}$. This led to extensions of the theory of Calder\'on and Zygmund in \cite{Calderon-Zygmund52}. In the case $p_0 = q_0$ it was shown that the smoothness condition on the kernel $K$ can be translated to a smoothness condition on the multiplier $m$ which is strong enough to deduce the classical Mihlin multiplier theorem. From here the field of harmonic analysis has quickly developed itself and this development is still ongoing. We refer to \cite{Grafakos08,Grafakos09,Katznelson04,Stein93} and references therein for a treatment and the history of the subject.

In the vector-valued setting it was shown in \cite{BeCaPa62} that the extrapolation results of H\"ormander for $p=q$ still holds. However, there is a catch:
\begin{itemize}
\item even for $p=q=2$ one does not have $T_m\in \La(\Elltwo(\Rd;X))$ for general $m\in \Ellinfty(\Rd)$ unless $X$ is a Hilbert space.
\end{itemize}
In \cite{Burkholder83} it was shown that $T_m\in \La(\Ellp(\R^d;X))$ for $m(\xi) := \text{sign}(\xi)$ if $X$ satisfies the so-called UMD condition. In \cite{Bourgain83} it was realized that this yields a characterization of the UMD property. In \cite{Bourgain86}, \cite{McConnell84}, \cite{Zimmermann89} versions of the Littlewood--Paley theorem and the Mihlin multiplier theorem were established in the UMD setting. These are very useful for operator theory and evolution equations (see for example \cite{Dore-Venni87}).

In the vector-valued setting it is rather natural to allow $m$ to take values in the space $\La(X,Y)$ of bounded operators from $X$ to $Y$. Pisier and Le Merdy showed that the natural analogues of the Mihlin multiplier theorem do not extend to this setting unless $X$ has cotype $2$ and $Y$ has type $2$ (a proof was published only later on in \cite{Arendt-Bu02}). On the other hand there was a need for such extensions as it was realized that multiplier theorems with operator-valued symbols are useful in the stability theory and the regularity theory for evolution equations (see \cite{Amann97, Hieber99, Weis97}). The missing ingredient for a natural analogue of the Mihlin multiplier theorem turned out to be $R$-boundedness, which is a strengthening of uniform boundedness (see \cite{BerksonGillespie94, ClPaSuWi00}). In \cite{Weis2001} it was shown that Mihlin's theorem holds for $m:\R\to \La(X)$ if the sets
\begin{align*}
\left\{m(\xi)\mid \xi\in\R\setminus\{0\}\right\} \ \text{and} \  \left\{\xi m'(\xi)\mid \xi\in \R\setminus\{0\}\right\}
\end{align*}
are $R$-bounded. Conversely, the $R$-boundedness of $\left\{m(\xi)\mid \xi\in \R\setminus\{0\}\right\}$ is also necessary. These results were used to characterize maximal $\Ellp$-regularity, and were then used by many authors in evolution equations, partial differential equations, operator theory and harmonic analysis (see the surveys and lecture notes \cite{Amann97, DeHiPr03, Kalton-Weis01, Kunstmann-Weis04}).
A generalization to multipliers on $\R^d$ instead of $\R$ was given in \cite{HaHeNo02} and \cite{Strkalj-Weis07}, but in some cases one additionally needs the so-called property $(\alpha)$ of the Banach space (which holds for all UMD lattices). Improvements of the multiplier theorems under additional geometric assumptions have been studied in \cite{Girardi-WeisJFA} and \cite{Shahmurov10} assuming Fourier type and in \cite{Hytonen10} assuming type and cotype conditions.

In this article we complement the theory of operator-valued Fourier multipliers by studying the boundedness of $T_{m}$ from $\Ell^{p}(\Rd;X)$ to $\Ell^{q}(\Rd;Y)$ for $p<q$. One of our main results is formulated under $\gamma$-boundedness assumptions on  $\{|\xi|^{\frac{d}{r}} m(\xi)\mid\xi\in \Rd\setminus\{0\}\}$. We note that $R$-boundedness implies $\gamma$-boundedness (see Subsection \ref{gamma-bddness}). The result is as follows (see Theorem \ref{thm:sharpintegrability} for the proof):

\begin{theorem}\label{thm:sharpintegrabilityintro}
Let $X$ be a Banach space with type $p_0\in(1,2]$ and $Y$ a Banach space with cotype $q_0\in[2,\infty)$, and let $p\in (1, p_0)$, $q\in (q_0, \infty)$. Let $r\in[1,\infty]$ be such that $\frac{1}{r}=\frac{1}{p}-\frac{1}{q}$. If $m:\Rd\setminus\{0\}\to\La(X,Y)$ is $X$-strongly measurable and
\begin{align}\label{eq:gammabddintro}
\{\abs{\xi}^{\frac{d}{r}}m(\xi)\mid \xi\in\Rd\setminus\{0\}\}\subseteq\La(X,Y)
\end{align}
is $\gamma$-bounded, then $T_{m}:\Ellp(\Rd;X)\to \Ellq(\Rd;Y)$ is bounded.
Moreover, if $p_0 = 2$ (or $q_0 = 2)$, then one can also take $p = 2$ (or $q=2$).
\end{theorem}

The condition $p\leq 2\leq q$ cannot be avoided in such results (see below \eqref{weakLrcondition}). Note that no smoothness on $m$ is required. Theorem \ref{thm:sharpintegrabilityintro} should be compared to the sufficient condition in \eqref{weakLrcondition} due to H\"{o}rmander in the case where $X=Y=\C$. We will give an example which shows that the $\gamma$-boundedness condition \eqref{eq:gammabddintro} cannot be avoided in general. Moreover, we obtain several converse results stating that type and cotype are necessary.

We note that, in case $m$ is scalar-valued and $X = Y$, the $\gamma$-boundedness assumption in Theorem \ref{thm:sharpintegrabilityintro} reduces to the uniform boundedness of \eqref{eq:gammabddintro}. Even in this setting of scalar multipliers our results appear to be new.

In Theorem \ref{multiplier Bessel spaces} we obtain a variant of Theorem \ref{thm:sharpintegrabilityintro} for $p$-convex and $q$-concave Banach lattices, where one can take $p = p_0$ and $q=q_0$. In \cite{Rozendaal-Veraar16Besov} we will deduce multiplier results similar to Theorem \ref{thm:sharpintegrabilityintro} in the Besov scale, where one can let $p = p_0$ and $q=q_0$ for Banach spaces $X$ and $Y$ with type $p$ and cotype $q$.

A vector-valued generalization of \eqref{weakLrcondition} is presented in Theorem \ref{Lp-Lq multipliers Fourier type2}. We show that if $X$ has Fourier type $p_{0}>p$ and $Y$ has Fourier type $q_{0}'>q'$, then
\[
\|T_{m}\|_{\La(\Ellp(\Rd;X),\Ellq(\Rd;Y))}\leq C\norm{\|m(\cdot)\|_{\La(X,Y)}}_{\Ell^{r,\infty}(\Rd)},
\]
where $\frac1r = \frac1p - \frac1q$. We show that in this result the Fourier type assumption is necessary. It should be noted that for many spaces (including all $\Ellr$-spaces for $r\in[1,\infty)\setminus\{2\}$), working with Fourier type yields more restrictive results in terms of the underlying parameters than working with type and cotype (see Subsection \ref{subsec:typecotype} for a discussion of the differences between Fourier type and (co)type).

The exponents $p$ and $q$ in Theorem \ref{thm:sharpintegrabilityintro} are fixed by the geometry of the underlying Banach spaces. However, Corollary \ref{cor:extrapol m uniform} shows that under smoothness conditions on the multiplier, one can extend the boundedness result to all pairs $(\tilde{p},\tilde{q})$ satisfying $1<\tilde{p}\leq \tilde{q}<\infty$ and $\frac{1}{\tilde{p}} - \frac{1}{\tilde{q}} = \frac{1}{p} - \frac{1}{q} = \frac1r$. Here the required smoothness depends on the Fourier type of $X$ and $Y$ and on the number $r\in(1,\infty]$.  We note that even in the case where $X=Y=\C$, for $p<q$ we require less smoothness for the extrapolation than in the classical results (see Remark \ref{rem:less derivatives}).

We will mainly consider multiplier theorems on $\R^d$. There are two exceptions. In Remark \ref{rem:more general groups} we deduce a result for more general locally compact groups. Moreover, in Proposition \ref{prop:transference} we show how to transfer our results from $\R^d$ to the torus $\T^d$.  This result appears to be new even in the scalar setting. As an application of the latter we show that certain irregular Schur multipliers with sufficient decay are bounded on the Schatten class $\Cs^p$ for $p\in(1,\infty)$.

We have pointed out that questions about operator-valued Fourier multiplier theorems were originally motivated by  stability and regularity theory. We have already successfully applied our result to stability theory of $C_0$-semigroups, as will be presented in a forthcoming paper \cite{Rozendaal-Veraar17Stability}. In \cite{Rozendaal15b} the first-named author has also applied the Fourier multiplier theorems in this article to study the $\mathcal{H}^{\infty}$-calculus for generators of $C_{0}$-groups.

Other potential applications could be given to the theory of dispersive equations. For instance the classical Strichartz estimates can be viewed as operator-valued $\Ellp$-$\Ellq$-multiplier theorems. Here the multipliers are often not smooth, as is the case in our theory. More involved applications probably require extensions of our work to oscillatory integral operators, which would be a natural next step in the research on  vector-valued singular integrals from $\Ellp$ to $\Ellq$.

This article is organized as follows. In Section \ref{preliminaries} we discuss some preliminaries on the geometry of Banach spaces and on function space theory. In Section \ref{Fourier multipliers} we introduce Fourier multipliers and prove our main results on $\Ellp$-$\Ellq$-multipliers in the vector-valued setting. In Section \ref{extrapolation} we present an extension of the extrapolation result under H\"ormander--Mihlin conditions to the case $p\leq q$.

\subsection{Notation and terminology}

We write $\N:=\left\{1,2,3,\ldots\right\}$ for the natural numbers and $\N_{0}:=\N\cup\left\{0\right\}$.

We denote nonzero Banach spaces over the complex numbers by $X$ and $Y$. The space of bounded linear operators from $X$ to $Y$ is $\La(X,Y)$, and $\La(X):=\La(X,X)$. The identity operator on $X$ is denoted by $\mathrm{I}_{X}$.

For $p\in[1,\infty]$ and $(\Omega,\mu)$ a measure space, $\Ellp(\Omega;X)$ denotes the Bochner space of equivalence classes of strongly measurable, $p$-integrable, $X$-valued functions on $\Omega$. Moreover, $\Ell^{p,\infty}(\Omega;X)$ is the weak $\Ellp$-space of all $f:\Omega\to X$ for which
\begin{align}\label{eq:weakLpnotation}
\|f\|_{\Ell^{p,\infty}(\Omega;X)} := \sup_{\alpha>0} \alpha \lambda_f(\alpha)^{\frac1p}<\infty,
\end{align}
where $\lambda_f(\alpha) := \mu(\{s\in \Omega\mid\|f(s)\|_{X}>\alpha\})$ for $\alpha>0$.
In the case where $\Omega\subseteq\Rd$ we implicitly assume that $\mu$ is the Lebesgue measure. Often we will use the shorthand notations $\|\cdot\|_{p}$ and $\|\cdot\|_{p,\infty}$ for the $\Ellp$-norm and $\Ell^{p,\infty}$-norm.

The H\"{o}lder conjugate of $p$ is denoted by $p'$ and is defined by $1=\frac{1}{p}+\frac{1}{p'}$. We write $\ell^{p}$ for the space of $p$-summable sequences $(x_{k})_{k\in\N_{0}}\subseteq\C$, and denote by $\ell^{p}(\Z)$ the space of $p$-summable sequences $(x_{k})_{k\in\Z}\subseteq\C$.

We say that a function $m:\Omega\to\La(X,Y)$ is \emph{$X$-strongly measurable} if $\omega\mapsto m(\omega)x$ is a strongly measurable $Y$-valued map for all $x\in X$. We often identify a scalar function $m:\Rd\to\C$ with the operator-valued function $\widetilde{m}:\Rd\to\La(X)$ given by $\widetilde{m}(\xi):=m(\xi)\mathrm{I}_{X}$ for $\xi\in\Rd$.

The class of $X$-valued rapidly decreasing smooth functions on $\Rd$ (the Schwartz functions) is denoted by $\Sw(\Rd;X)$, and the space of $X$-valued tempered distributions by $\Sw'(\Rd;X)$. We write $\Sw(\Rd):=\Sw(\Rd;\C)$ and denote by $\lb \cdot,\cdot\rb:\Sw'(\Rd;X)\times \Sw(\Rd)\to X$ the $X$-valued duality between $\Sw'(\Rd;X)$ and $\Sw(\Rd)$. The Fourier transform of a $\Phi\in\Sw'(\Rd;X)$ is denoted by $\F \Phi$ or $\widehat{\Phi}$. If
$f\in\Ell^{1}(\Rd;X)$ then
\begin{align*}
\widehat{f}(\xi)=\F f(\xi):=\int_{\Rd}\ue^{-2\pi\ui \xi \cdot t}f(t)\,\ud t\qquad(\xi\in\Rd).
\end{align*}

A standard complex Gaussian random variable is a random variable $\gamma:\Omega\to\C$ of the form $\gamma=\frac{\gamma_{r}+\ui \gamma_{i}}{\sqrt{2}}$, where $(\Omega,\mathbb{P})$ is a probability space and $\gamma_{r},\gamma_{i}:\Omega\to\R$ are independent standard real Gaussians. A \emph{Gaussian sequence} is a (finite or infinite) sequence $(\gamma_{k})_{k}$ of independent standard complex Gaussian random variables on some probability space.

We will use the convention that a constant $C$ which appears multiple times in a chain of inequalities may vary from one occurrence to the next.

\section{Preliminaries}\label{preliminaries}

\subsection{Fourier type}\label{Fourier type}

We recall some background on the Fourier type of a Banach space. For these facts and for more on Fourier type see \cite{GaKaKoTo98, Pietsch-Wenzel98, HyNeVeWe16}.

A Banach space $X$ has \emph{Fourier type} $p\in[1,2]$ if the Fourier transform $\F$ is bounded from $\Ellp(\Rd;X)$ to $\Ellpprime(\Rd;X)$ for some (in which case it holds for all) $d\in\N$. We then write $\F_{p,X,d}:=\|\F\|_{\La(\Ellp(\Rd;X),\Ellpprime(\Rd;X))}$.

Each Banach space $X$ has Fourier type $1$ with $\F_{1,X,d}=1$ for all $d\in\N$. If $X$ has Fourier type $p\in[1,2]$ then $X$ has Fourier type $r$ with $\F_{r,X,d}\leq \F_{p,X,d}$ for all $r\in[1,p]$  and $d\in\N$. We say that $X$ has \emph{nontrivial Fourier type} if $X$ has Fourier type $p$ for some $p\in(1,2]$. In order to make our main results more transparent we will say that $X$ has Fourier cotype $p'$ whenever $X$ has Fourier type $p$.

Let $X$ be a Banach space, $r\in[1,\infty)$ and let $\Omega$ be a measure space. If $X$ has Fourier type $p\in[1,2]$ then $\Ellr(\Omega;X)$ has Fourier type $\min(p,r,r')$. In particular, $\Ellr(\Omega)$ has Fourier type $\min(r,r')$.

\subsection{Type and cotype\label{subsec:typecotype}}

We first recall some facts concerning the type and cotype of Banach spaces. For more on these notions and for unexplained results see \cite{Albiac-Kalton06}, \cite{DiJaTo95}, \cite{HyNeVeWe2} and \cite[Section 9.2]{Lindenstrauss-Tzafriri79}.

Let $X$ be a Banach space, $(\gamma_{n})_{n\in\N}$ a Gaussian sequence on a probability space $(\Omega,\mathbb{P})$ and let $p\in[1,2]$ and $q\in[2,\infty]$. We say that $X$ has \emph{(Gaussian) type} $p$ if there exists a constant $C\geq 0$ such that for all $m\in\N$ and all $x_{1},\ldots, x_{m}\in X$,
\begin{align}\label{type}
\Big(\mathbb{E}\Big\|\sum_{n=1}^{m}\gamma_{n}x_{n}\Big\|^{2}\Big)^{1/2}\leq C\Big(\sum_{n=1}^{m}\|x_{n}\|^{p}\Big)^{1/p}.
\end{align}
We say that $X$ has \emph{(Gaussian) cotype} $q$ if there exists a constant $C\geq 0$ such that for all $m\in\N$ and all $x_{1},\ldots, x_{m}\in X$,
\begin{align}\label{cotype}
\Big(\sum_{n=1}^{m}\|x_{n}\|^{q}\Big)^{1/q}\leq C\Big(\mathbb{E}\Big\|\sum_{n=1}^{m}\gamma_{n}x_{n}\Big\|^{2}\Big)^{1/2},
\end{align}
with the obvious modification for $q=\infty$.

The minimal constants $C$ in \eqref{type} and \eqref{cotype} are called the \emph{(Gaussian) type $p$ constant} and the \emph{(Gaussian) cotype $q$ constant} and will be denoted by $\tau_{p,X}$ and $c_{q,X}$. We say that $X$ has \emph{nontrivial type} if $X$ has type $p\in(1,2]$, and \emph{finite cotype} if $X$ has cotype $q\in[2,\infty)$.

Note that it is customary to replace the Gaussian sequence in \eqref{type} and \eqref{cotype} by a \emph{Rademacher sequence}, i.e.\ a sequence $(r_{n})_{n\in\N}$ of independent random variables on a probability space $(\Omega,\mathbb{P})$ that are uniformly distributed on $\{z\in\R\mid \abs{z}=1\}$. This does not change the class of spaces under consideration, only the minimal constants in \eqref{type} and \eqref{cotype} (see \cite[Chapter 12]{DiJaTo95}). We choose to work with Gaussian sequences because the Gaussian constants $\tau_{p,X}$ and $c_{q,X}$ occur naturally here.

Each Banach space $X$ has type $p=1$ and cotype $q=\infty$, with $\tau_{1,X}=c_{\infty,X}=1$. If $X$ has type $p$ and cotype $q$ then $X$ has type $r$ with $\tau_{r,X}\leq \tau_{p,X}$ for all $r\in[1,p]$ and cotype $s$ with $c_{s,X}\leq c_{q,X}$ for all $s\in[q,\infty]$. A Banach space $X$ is isomorphic to a Hilbert space if and only if $X$ has type $p=2$ and cotype $q=2$, by Kwapie\'n's theorem (see \cite[Theorem 7.4.1]{Albiac-Kalton06}). Also, a Banach space $X$ with nontrivial type has finite cotype by the Maurey--Pisier theorem (see \cite[Theorem 11.1.14]{Albiac-Kalton06}).

Let $X$ be a Banach space, $r\in[1,\infty)$ and let $\Omega$ be a measure space. If $X$ has type $p\in[1,2]$ and cotype $q\in[2,\infty)$ then $\Ellr(\Omega;X)$ has type $\min(p,r)$ and cotype $\max(q,r)$ (see \cite[Theorem 11.12]{DiJaTo95}).

A Banach space with Fourier type $p\in[1,2]$ has type $p$ and cotype $p'$ (see \cite{HyNeVeWe2}). By a result of Bourgain a Banach space has nontrivial type if and only if it has nontrivial Fourier type (see \cite[5.6.30]{Pietsch-Wenzel98}).

\subsection{Convexity and concavity}\label{convexity and concavity}

For the theory of Banach lattices we refer the reader to \cite{Lindenstrauss-Tzafriri79}. We repeat some of the definitions which will be used frequently.

Let $X$ be a Banach lattice and $p,q\in[1,\infty]$. We say that $X$ is \emph{$p$-convex} if there exists a constant $C\geq 0$ such that for all $n\in\N$ and all $x_{1},\ldots, x_{n}\in X$,
\begin{align}\label{p-convex}
\Big\|\Big(\sum_{k=1}^{n}\abs{x_{k}}^{p}\Big)^{1/p}\Big\|_{X}\leq C\Big(\sum_{k=1}^{n}\|x_{k}\|_{X}^{p}\Big)^{1/p},
\end{align}
with the obvious modification for $p=\infty$. We say that $X$ is \emph{$q$-concave} if there exists a constant $C\geq 0$ such that for all $n\in\N$ and all $x_{1},\ldots, x_{n}\in X$,
\begin{align}\label{q-concave}
\Big(\sum_{k=1}^{n}\|x_{k}\|_{X}^{q}\Big)^{1/q}\leq C\Big\|\Big(\sum_{k=1}^{n}\abs{x_{k}}^{q}\Big)^{1/q}\Big\|_{X},
\end{align}
with the obvious modification for $q=\infty$.

Every Banach lattice $X$ is $1$-convex and $\infty$-concave. If $X$ is $p$-convex and $q$-concave then it is $r$-convex and $s$-concave for all $r\in[1,p]$ and $s\in[q,\infty]$. By \cite[Proposition 1.f.3]{Lindenstrauss-Tzafriri79}, if $X$ is $q$-concave then it has cotype $\max(q,2)$, and if $X$ is $p$-convex and $q$-concave for some $q<\infty$ then $X$ has type $\min(p,2)$.

If $X$ is $p$-convex and $p'$-concave for $p\in[1,2]$ then $X$ has Fourier type $p$, by \cite[Proposition 2.2]{GTK96}. For $(\Omega,\mu)$ a measure space and $r\in[1,\infty)$, $\Ellr(\Omega,\mu)$ is an $r$-convex and $r$-concave Banach lattice. Moreover, if $X$ is $p$-convex and $q$-concave and $r\in [1, \infty)$, then $\Ellr(\Omega;X)$ is $\min(p,r)$-convex and $\max(q,r)$-concave.

Specific Banach lattices which we will consider are the \emph{Banach function spaces}. For the definition and details of these spaces we refer to \cite{Lin04}. If $X$ is a Banach function space over a measure space $(\Omega,\mu)$ and $Y$ is a Banach space, then $X(Y)$ consists of all $f:\Omega\to Y$ such that $\|f(\cdot)\|_{Y}\in X$, with the norm
\begin{align*}
\|f\|_{X(Y)}:=\norm{\|f(\cdot)\|_{Y}}_{X}\qquad(f\in X(Y)).
\end{align*}

If $f\in X(\Ellp(\Rd))$ for $p\in[1,\infty)$ and $d\in\N$ then we write $(\int_{\Rd}\abs{f(t)}^{p}\,\ud t)^{1/p}$ for the element of $X$ given by
\begin{align*}
\Big(\int_{\Rd}\abs{f(t)}^{p}\,\ud t\Big)^{1/p}(\w):=\Big(\int_{\Rd}\abs{f(\omega)(t)}^{p}\,\ud t\Big)^{1/p}\qquad(\omega\in\Omega).
\end{align*}
Note that $\|f\|_{X(\Ellp(\Rd))}=\|(\int_{\Rd}\abs{f(t)}^{p}\,\ud t)^{1/p}\|_{X}$

Let $f=\sum_{k=1}^{n}f_{k}\otimes x_{k}\in \Ellp(\Rd)\otimes X$, for $n\in\N$, $f_{1},\ldots, f_{n}\in \Ellp(\Rd)$ and $x_{1},\ldots, x_{n}\in X$. Then $f$ determines both an element $[t\mapsto \sum_{k=1}^{n}f_{k}(t)x_{k}]$ of $\Ellp(\Rd;X)$ and an element $[\omega\mapsto \sum_{k=1}^{n}x_{k}(\omega)f_{k}]$ of $X(\Ellp(\Rd))$. Throughout we will identify these and consider $f$ as an element of both $\Ellp(\Rd;X)$ and $X(\Ellp(\Rd))$. The following lemma, proved as in \cite[Theorem 3.9]{Veraar13} by using \eqref{p-convex} and \eqref{q-concave} on simple $X$-valued functions and then approximating, relates the $\Ellp(\Rd;X)$-norm and the $X(\Ellp(\Rd))$-norm of such an $f$ and will be used later.

\begin{lemma}\label{lem:relation different p-norms}
Let $X$ be a Banach function space, $p\in[1,\infty)$ and $f\in\Ellp(\Rd)\otimes X$.

\begin{itemize}
\item If $X$ is $p$-convex then
\begin{align*}
\|f\|_{X(\Ellp(\Rd))}\leq C\|f\|_{\Ellp(\Rd;X)},
\end{align*}
where $C\geq 0$ is as in \eqref{p-convex}.
\item If $X$ is $p$-concave then
\begin{align*}
\|f\|_{\Ellp(\Rd;X)}\leq C\|f\|_{X(\Ellp(\Rd))},
\end{align*}
where $C\geq 0$ is as in \eqref{q-concave}
\end{itemize}
\end{lemma}
\vanish{
\begin{proof}
First suppose that $f=\sum_{k=1}^{n}\ind_{A_{k}}\otimes x_{k}$ for $n\in\N$, $A_{1},\ldots, A_{n}\subseteq\Rd$ measurable and disjoint, and $x_{1},\ldots, x_{n}\in X$. Then
\begin{align*}
\|f\|_{X(\Ellp(\Rd))}&=\Big\|\Big(\int_{\Rd}\Big|\sum_{k=1}^{n}\ind_{A_{k}}(t) x_{k}\Big|^{p}\,\ud t\Big)^{1/p}\Big\|_{X}=\Big\|\Big(\sum_{k=1}^{n}\abs{A_{k}} \abs{x_{k}}^{p}\Big)^{1/p}\Big\|_{X}\\
&\leq C\Big(\sum_{k=1}^{n}\abs{A_{k}}\,\|x_{k}\|_{X}^{p}\Big)^{1/p}=C\|f\|_{\Ellp(\Rd;X)},
\end{align*}
by \eqref{p-convex}.

Now let $f=\sum_{k=1}^{n}f_{k}\otimes x_{k}$ for $n\in\N$, $f_{1},\ldots, f_{n}\in \Ellp(\Rd)$, $x_{1},\ldots, x_{n}\in X$. For $1\leq k\leq n$, let $(s_{k,m})_{m\in\N}\subseteq\Ellp(\Rd)$ be a sequence of simple functions such that $s_{k,m}\to f_{k}$ in $\Ellp(\Rd)$ as $m\to\infty$. For $m\in\N$ let $g_{m}:=\sum_{k=1}^{m}s_{k,m}\otimes x_{k}$. Then $g_{m}$ is an $X$-valued simple function. Moreover,
\begin{align*}
\|f-g_{m}\|_{\Ellp(\Rd;X)}&=\|\sum_{k=1}^{n}(f_{k}-s_{k,m})\otimes x_{k}\|_{\Ellp(\Rd;X)}\\
&\leq \sum_{k=1}^{n}\|f_{k}-s_{k,m}\|_{\Ellp(\Rd)}\| x_{k}\|_{X}\to0
\end{align*}
as $m\to\infty$. Similarly, $\|f-g_{m}\|_{X(\Ellp(\Rd))}\to0$ as $m\to\infty$. Hence, by what we have already shown,
\begin{align*}
\|f\|_{X(\Ellp(\Rd))}=\lim_{m\to\infty}\|g_{m}\|_{X(\Ellp(\Rd))}\leq C\|g_{m}\|_{\Ellp(\Rd;X)}=C\|f\|_{\Ellp(\Rd;X)}.
\end{align*}
\end{proof}
}

The proof of the following lemma is the same as in \cite[Lemma 4]{Montgomery-Smith96} for simple $X$-valued functions, and the general case follows by approximation.

\begin{lemma}\label{lem:positive operator through integral}
Let $X$ and $Y$ be Banach function spaces, $P\in\La(X,Y)$ a positive operator, $p\in[1,\infty)$ and $f\in\Ellp(\Rd)\otimes X$. Then
\begin{align*}
\Big(\int_{\Rd}\abs{P(f(t))}^{p}\,\ud t\Big)^{1/p}\leq P\left( \Big(\int_{\Rd}\abs{f(t)}^{p}\,\ud t\Big)^{1/p}\right).
\end{align*}
\end{lemma}
\vanish{
\begin{proof}
The proof for $f$ a simple $X$-valued function is as in \cite[Lemma 4]{Montgomery-Smith}. For a general $f\in\Ellp(\Rd)\otimes X$, let  $n\in\N$, $f_{1},\ldots, f_{n}\in \Ellp(\Rd)$ and $x_{1},\ldots, x_{n}\in X$ be such that $f=\sum_{k=1}^{n}f_{k}\otimes x_{k}$. For $1\leq k\leq n$, let $(s_{k,m})_{m\in\N}\subseteq\Ellp(\Rd)$ be a sequence of simple functions such that $s_{k,m}\to f_{k}$ in $\Ellp(\Rd)$ as $m\to\infty$. For $m\in\N$ let $g_{m}:=\sum_{k=1}^{m}s_{k,m}\otimes x_{k}$. Then $g_{m}$ is an $X$-valued simple function. Moreover, $g_{m}\to f$ in $X(\Ellp(\Rd))$. Hence
\begin{align*}
\Big(\int_{\Rd}\abs{g_{m}(t)}^{p}\,\ud t\Big)^{1/p}=\|g_{m}\|_{\Ellp(\Rd)}\to \|f\|_{\Ellp(\Rd)}=\Big(\int_{\Rd}\abs{f(t)}^{p}\,\ud t\Big)^{1/p}
\end{align*}
as a limit in $X$. Since $P$ is bounded,
\begin{align*}
P\left(\Big(\int_{\Rd}\abs{g_{m}(t)}^{p}\,\ud t\Big)^{1/p}\right)\to P\left(\Big(\int_{\Rd}\abs{f(t)}^{p}\,\ud t\Big)^{1/p}\right)
\end{align*}
in $Y$ as $m\to\infty$. Also, $Pg_{m}=\sum_{k=1}^{n}s_{k,m}\otimes Px_{k}\to\sum_{k=1}^{n}f_{k}\otimes Px_{k}=Pf$ in $Y(\Ellp(\Rd))$. We now conclude that
\begin{align*}
\Big(\int_{\Rd}\abs{Pf(t)}^{p}\,\ud t\Big)^{1/p}&=\lim_{m\to\infty}\Big(\int_{\Rd}\abs{Pg_{m}}^{p}\,\ud t\Big)^{1/p}\leq \lim_{m\to\infty}P\Big(\Big(\int_{\Rd}\abs{g_{m}}^{p}\,\ud t\Big)^{1/p}\Big)\\
&=P\Big(\Big(\int_{\Rd}\abs{f}^{p}\,\ud t\Big)^{1/p}\Big),
\end{align*}
which concludes the proof.
\end{proof}
}

\subsection{$\gamma$-boundedness}\label{gamma-bddness}
Let $X$ and $Y$ be Banach spaces. A collection $\mathcal{T}\subseteq \La(X,Y)$ is \emph{$\gamma$-bounded}
if there exists a constant $C\geq 0$ such that
\begin{align}\label{gamma-boundedness}
\Big(\mathbb{E}\Big\|\sum_{k=1}^{n}\gamma_{k}T_{k}x_{k}\Big\|_{Y}^{2}\Big)^{1/2}\leq
C\Big(\mathbb{E}\Big\|\sum_{k=1}^{n}\gamma_{k}x_{k}\Big\|_{X}^{2}\Big)^{1/2}
\end{align}
for all $n\in\N$, $T_{1},\ldots, T_{n}\in\mathcal{T}$, $x_{1},\ldots, x_{n}\in X$ and each Gaussian sequence $(\gamma_{k})_{k=1}^{n}$. The smallest such $C$ is the \emph{$\gamma$-bound} of $\mathcal{T}$ and is denoted by $\gamma(\mathcal{T})$. By the Kahane-Khintchine inequalities, we may replace the $\Elltwo$-norm in \eqref{gamma-boundedness} by an $\Ellp$-norm for each $p\in[1,\infty)$.

Every $\gamma$-bounded collection is uniformly bounded with supremum bound less than or equal to the $\gamma$-bound, and the converse holds if and only if $X$ has cotype $2$ and $Y$ has type $2$ (see \cite{Arendt-Bu02}).
By the Kahane contraction principle, for each $\gamma$-bounded collection $\mathcal{T}\subseteq\La(X,Y)$ and each $\lambda\in[0,\infty)$, the closure in the strong operator topology of the family $\{z T\mid z\in\C, \abs{z}\leq \lambda, T\in\mathcal{T}\}\subseteq\La(X,Y)$ is $\gamma$-bounded with
\begin{align}\label{Kahane contraction principle}
\gamma\Big(\overline{\{z T\mid z\in\C, \abs{z}\leq \lambda, T\in\mathcal{T}\}}^{\text{SOT}}\Big)\leq \lambda\gamma(\mathcal{T}).
\end{align}

By replacing the Gaussian random variables in \eqref{gamma-boundedness} by Rademacher variables, one obtains the definition of an \emph{$R$-bounded} collection $\mathcal{T}\subseteq\La(X,Y)$. Each $R$-bounded collection is $\gamma$-bounded. The notions of $\gamma$-boundedness and $R$-boundedness are equivalent if and only if $X$ has finite cotype (see \cite[Theorem 1.1]{KwVeWe14}), but the minimal constant $C$ in \eqref{gamma-boundedness} may depend on whether one considers Gaussian or Rademacher variables. In this article we work with $\gamma$-boundedness instead of $R$-boundedness because in our results we will allow spaces which do not have finite cotype.

\subsection{Bessel spaces}\label{function spaces}

For details on Bessel spaces and related spaces see e.g.~\cite{Amann97,Bergh-Lofstrom76, HyNeVeWe16,Triebel10}.

For $X$ a Banach space, $s\in\R$ and $p\in[1,\infty]$ the \emph{inhomogeneous Bessel potential space} $\Hr^{s}_{p}(\Rd;X)$ consists of all $f\in\Sw'(\Rd;X)$ such that $\F^{-1}((1+\abs{\cdot})^{s/2}\widehat{f}(\cdot)\,)\in\Ellp(\Rd;X)$. Then $\Hr^{s}_{p}(\Rd;X)$ is a Banach space endowed with the norm
\begin{align*}
\|f\|_{\Hr^{s}_{p}(\Rd;X)}:=\|\F^{-1}((1+\abs{\cdot}^2)^{s/2}\widehat{f}(\cdot))\|_{\Ellp(\Rd;X)}\quad(f\in\Hr^{s}_{p}(\Rd;X)),
\end{align*}
and $\Sw(\Rd;X)\subseteq\Hr^{s}_{p}(\Rd;X)$ lies dense if $p<\infty$.

\medskip

In this article we will also deal with homogeneous Bessel spaces. To define these spaces we follow the approach of
\cite[Chapter 5]{Triebel10} (see also \cite{Triebel15}). Let $X$ be a Banach space and define
\begin{align*}
\dot{\Sw}(\Rd;X):=\{f\in\Sw(\Rd;X)\mid \mathrm{D}^{\alpha}\!\widehat{f}(0)=0\text{ for all } \alpha\in\N_{0}^{d}\}.
\end{align*}
Endow $\dot{\Sw}(\Rd;X)$ with the subspace topology induced by $\Sw(\Rd;X)$ and set $\dot{\Sw}(\Rd):=\dot{\Sw}(\Rd;\C)$. Let $\dot{\Sw}'(\Rd;X)$ be the space of continuous linear mappings $\dot{\Sw}(\Rd)\to X$. Then each $f\in \Sw'(\Rd;X)$ yields an $f\restriction_{\dot{\Sw}(\Rd)}\in\dot{\Sw}'(\Rd;X)$ by restriction, and $f\restriction_{\dot{\Sw}(\Rd)}=g\restriction_{\dot{\Sw}(\Rd)}$ if and only if $\supp(\widehat{f}-\widehat{g})\subseteq\{0\}$. Conversely, one can check that each $f\in \dot{\Sw}'(\Rd;X)$ extends to an element of $\Sw'(\Rd;X)$ (see \cite{Rozendaal-Veraar16Besov} for the tedious details in the vector-valued setting). Hence $\dot{\Sw}'(\Rd;X)=\Sw'(\R^d;X)/\Pa(\Rd;X)$ for $\Pa(\Rd;X):=\{f\in\Sw'(\Rd;X)\mid \supp(\widehat{f})\subseteq\{0\}\}$. As in \cite[Proposition 2.4.1]{Grafakos08} one can show that $\Pa(\Rd;X)=\Pa(\Rd)\otimes X$, where $\Pa(\Rd)$ is the collection of polynomials on $\Rd$. If $F(\Rd;X)\subseteq\Sw'(\Rd;X)$ is a linear subspace such that $\Phi=0$ if $\supp(\widehat{\Phi}\,)\subseteq\{0\}$, then we will identify $F(\Rd;X)$ with its image in $\dot{\Sw}'(\Rd;X)$. In particular, this is the case if $F(\Rd;X)=\Ellp(\Rd;X)$ for some $p\in[1,\infty]$.

For $s\in\R$ and $p\in[1,\infty]$, the \emph{homogeneous Bessel potential space} $\dot{\Hr}^{s}_{p}(\Rd;X)$ is the space of all $f\in\dot{\Sw}'(\Rd;X)$ such that $\F^{-1}(\abs{\cdot}^{s}\widehat{f}(\cdot))\in\Ellp(\Rd;X)$, where
\begin{align*}
\lb \F^{-1}(\abs{\cdot}^{s}\widehat{f}(\cdot)),\ph\rb:=\lb f,\F^{-1}(\abs{\cdot}^{s}\widehat{\ph}(\cdot))\rb\qquad (\ph\in\dot{\Sw}(\Rd;X)).
\end{align*}
Then $\dot{\Hr}^{s}_{p}(\Rd;X)$ is a Banach space endowed with the norm
\begin{align*}
\|f\|_{\dot{\Hr}^{s}_{p}(\Rd;X)}:=\|\F^{-1}(\abs{\cdot}^{s}\widehat{f}(\cdot))\|_{\Ellp(\Rd;X)}\qquad(f\in\dot{\Hr}^{s}_{p}(\Rd;X)),
\end{align*}
and $\dot{\Sw}(\Rd)\otimes X\subseteq\dot{\Hr}^{s}_{p}(\Rd;X)$ lies dense if $p<\infty$.

\section{Fourier multipliers results}\label{Fourier multipliers}

In this section we introduce operator-valued Fourier multipliers acting on various vector-valued function spaces and discuss some of their properties. We start with some preliminaries and after that in Subsection \ref{subsec:transference} we prove a result that will allow us to transfer boundedness of multipliers on $\Rd$ to the torus $\mathbb{T}^{d}$. Then in Subsection \ref{Fourier type assumptions} we present some first simple results under Fourier type conditions. We return to our main multiplier results for spaces with type, cotype, $p$-convexity and $q$-concavity in Subsections \ref{sec:Bessel1} and \ref{subsec:convexconcave}.

\subsection{Definitions and basic properties}\label{definitions and basic properties}

Fix $d\in\N$, let $X$ and $Y$ be Banach spaces, and let $m:\Rd\to\La(X,Y)$ be $X$-strongly measurable. We say that $m$ is \emph{of moderate growth at infinity} if there exist a constant $\alpha\in(0,\infty)$ and a $g\in \Ell^1(\R^d)$ such that
\begin{align*}
(1+\abs{\xi})^{-\alpha} \|m(\xi)\|_{\La(X,Y)} \leq g(\xi) \qquad(\xi\in\Rd).
\end{align*}
For such an $m$, let $T_{m}:\Sw(\Rd;X)\to\Sw'(\Rd;Y)$ be given by
\begin{align*}
T_{m}(f):=\F^{-1}(m\cdot\widehat{f}\,)\qquad(f\in\Sw(\Rd;X)).
\end{align*}
We call $T_{m}$ the \emph{Fourier multiplier operator} associated with $m$ and we call $m$ the \emph{symbol} of $T_{m}$.

Let $p,q\in[1,\infty]$. We say that $m$ is a bounded \emph{$(\Ell^p(\Rd;X), \Ell^q(\Rd;Y))$-Fourier multiplier} if there exists a constant $C\in(0,\infty)$ such that $T_{m}(f)\in L^q(\Rd;Y)$ and
\begin{align*}
\|T_{m}(f)\|_{L^{q}(\Rd;Y)}\leq C\|f\|_{L^{p}(\Rd;X)}
\end{align*}
for all $f\in\Sw(\Rd;X)$. In the case $1\leq p<\infty$, $T_{m}$ extends uniquely to a bounded operator from $L^p(\Rd;X)$ to $L^q(\Rd;Y)$ which will be denoted by $\widetilde{T_{m}}$, and often just by $T_m$ when there is no danger of confusion. If $X=Y$ and $p=q$ then we simply say that $m$ is an $L^p(\Rd;X)$-Fourier multiplier.

We will also consider Fourier multipliers on homogeneous function spaces. Let $X$ and $Y$ be Banach spaces and let $m:\Rd\setminus\{0\}\to\La(X,Y)$ be $X$-strongly measurable. We say that $m:\Rd\setminus\{0\}\to\La(X,Y)$ is of \emph{moderate growth at zero and infinity} if there exist a constant $\alpha\in(0,\infty)$ and a $g\in \Ell^1(\R^d)$ such that
\begin{align*}
\abs{\xi}^{\alpha}(1+\abs{\xi})^{-2\alpha} \|m(\xi)\|_{\La(X,Y)} \leq g(\xi) \qquad(\xi\in\Rd).
\end{align*}
For such an $m$, let $\dot{T}_{m}:\dot{\Sw}(\Rd;X)\to\Sw'(\Rd;Y)$ be given by
\begin{align*}
\dot{T}_{m}(f):=\F^{-1}(m\cdot\widehat{f}\,)\qquad(f\in\dot{\Sw}(\Rd;X)),
\end{align*}
where $\dot{T}_{m}(f)\in\Sw'(\Rd;Y)$ is well-defined by definition of $\dot{\Sw}(\Rd;X)$. We use similar terminology as before to discuss the boundedness of $\dot{T}_{m}$. Often we will simply write $T_{m}=\dot{T}_{m}$, to simplify notation.

In later sections we will use the following lemma about approximation of multipliers, which can be proved as in \cite[Proposition 2.5.13]{Grafakos08}.

\begin{lemma}\label{approximation of multipliers}
Let $X$ and $Y$ be Banach spaces and $q\in[1,\infty]$. For each $n\in\N$ let $m_{n}:\Rd\to\La(X,Y)$ be $X$-strongly measurable, and let $m:\Rd\to\La(X,Y)$ be such that $m(\xi)x=\lim_{n\to\infty}m_{n}(\xi)x$ for all $x\in X$ and almost all $\xi\in\Rd$.
Suppose that there exist $\alpha>0$ and $g\in \Ell^1(\R^d)$ such that
\begin{align*}
(1+\abs{\xi})^{\alpha} \|m_{n}(\xi)\|_{\La(X,Y)} \leq g(\xi)
\end{align*}
for all $n\in\N$ and $\xi\in\Rd$. If $f\in\Sw(\Rd;X)$ is such that $T_{m_{n}}(f)\in\Ellq(\Rd;Y)$ for all $n\in\N$, and if $\displaystyle \liminf_{n\to\infty}\|T_{m_{n}}(f)\|_{\Ellq(\Rd;Y)}<\infty$, then $T_{m}(f)\in\Ellq(\Rd;Y)$ with
\begin{align*}
\|T_{m}(f)\|_{\Ellq(\Rd;Y)}\leq \liminf_{n\to\infty}\|T_{m_{n}}(f)\|_{\Ellq(\Rd;Y)}.
\end{align*}
\end{lemma}

The same result holds for $f\in \dot{\Sw}(\Rd;X)$ if instead we assume that there exist an $\alpha>0$ and $g\in \Ellone(\Rd)$ such that, for all $n\in\N$ and $\xi\in\Rd$,
\begin{align*}
|\xi|^{-\alpha} (1+\abs{\xi})^{2\alpha} \|m_{n}(\xi)\|_{\La(X,Y)} \leq g(\xi).
\end{align*}

The case of positive scalar-valued kernels plays a special role. An immediate consequence of \cite[Proposition 4.5.10]{Grafakos08} is:

\begin{proposition}[Positive kernels]\label{prop:positivekern}
Let $m:\Rd\setminus\{0\}\to \C$ have moderate growth at zero
and infinity. Suppose that $\dot{T}_m:\Ellp(\Rd)\to \Ellq(\Rd)$ is bounded for some $p,q\in[1,\infty]$ and that $\F^{-1} m\in \dot{\Sw}'(\Rd)$ is positive. Then, for any Banach space $X$, the operator $T_m\otimes I_X:\Ell^p(\R^d;X)\to \Ell^q(\R^d;X)$ is bounded of norm
\begin{align*}
\|T_m\otimes I_{X}\|_{\La(\Ellp(\Rd;X),\Ellq(\Rd;Y))}\leq \|T_m\|_{\La(\Ellp(\Rd),\Ellq(\Rd))}.
\end{align*}
\end{proposition}

The Hardy--Littlewood--Sobolev inequality on fractional integration is a typical example where Proposition \ref{prop:positivekern} can be applied.

\begin{example}\label{ex:HLS}
Let $X$ be a Banach space and $1<p\leq q<\infty$. Let $m(\xi) := |\xi|^{-s}$ for $\xi\in\Rd$. Then $T_m:\Ell^p(\R^d;X)\to \Ell^q(\R^d;X)$ is bounded if and only if $\frac{1}{p} - \frac1q = \frac{s}{d}$. In this case $\F^{-1} m(\cdot) = C_s |\cdot|^{-d+s}$ is positive and therefore the result follows from the scalar case
(see \cite[Theorem 6.1.3]{Grafakos09}) and Proposition \ref{prop:positivekern}. The same holds for the multiplier $m(\cdot) := (1+|\cdot|^{2})^{-s/2}$ under the less restrictive condition $\frac{1}{p} - \frac1q \leq  \frac{s}{d}$.
\end{example}

\subsection{Transference from $\R^d$ to $\T^d$}\label{subsec:transference}

We will mainly consider Fourier multipliers on $\R^d$. However, we want to present at least one transference result to obtain Fourier multiplier results for the torus $\T^d := [0,1]^d$. The transference technique differs slightly from the standard setting of de Leeuw's theorem where $p=q$ (see \cite[Theorem 4.5]{Leeuw81} and \cite[Chapter 5]{HyNeVeWe16}), due to the fact that $\|T_{m_a}\|_{\calL(\Ell^p(\R^d),\Ell^q(\R^d))} = a^{-d/r} \|T_m\|_{\La(\Ellp(\R^d), \Ellq(\R^d))}$, where $\frac1r = \frac1p - \frac1q$ and $m_a(\xi) := m(a\xi)$ for $a>0$.

Let $e_k:\T^d\to \C$ be given by $e_k(t) := e^{2\pi i k \cdot t}$ for $k\in\Z$ and $t\in\T^{d}$.

\begin{proposition}[Transference]\label{prop:transference}
Let $p,q,r\in(1,\infty)$ be such that $\frac{1}{r} = \frac1p - \frac1q$. Let $m:\R^d\to \calL(X,Y)$
be such that $m(\cdot)x\in \Ell^1_{\rm loc}(\R^d;Y)$ for all $x\in X$. Fix $a>0$ and let $m_k x := a^{-d}\int_{[0,a]^d} m(t+ka)x \,\ud t$ for $k\in \Z^{d}$.
If $T_m:\Ell^p(\R^d;X)\to \Ell^q(\R^d;Y)$ is bounded, then for all $n\in \N$ and $(x_k)_{|k|\leq n}$ in $X$,
\[
a^{d/r}  \Big\|\sum_{|k|\leq n} e_k m_k x_k\Big\|_{\Ellq(\T^d;Y)}  \leq C_{d,p,q'} \|T_m\| \, \Big\|\sum_{|k|\leq n} e_k x_k \Big\|_{\Ellp(\T^d;X)}
\]
for some $C_{d,p,q'}\geq 0$. In particular, the Fourier multiplier operator with symbol $(m_{k})_{k\in\Z^{d}}$ is bounded from $\Ellp(\T^{d};X)$ to $\Ellq(\T^{d};Y)$.
\end{proposition}
This result seems to be new even in the scalar case $X=Y=\C$.
\begin{proof}
Let $P = \sum_{|k|\leq n} e_k x_k$. Since $\Ell^{q'}(\T^d;Y^*)$ is norming for $\Ell^{q}(\T^d;Y)$ and  since
the $Y^{*}$-valued trigonometric polynomials are dense in $\Ell^{q'}(\T^d;Y^*)$, it suffices to show that
\begin{equation}\label{eq:toprovebddTorus}
a^{d/r}  \Big|\Big<\sum_{|k|\leq n}  e_k m_k x_k, Q\Big>\Big|  \leq C_{d,p,q'} \|T_m\| \, \|P\|_{\Ellp(\T^{d};X)} \|Q\|_{\Ell^{q'}(\T^{d};Y^*)}
\end{equation}
for $Q:\T^d\to Y^*$ an arbitrary $Y^{*}$-valued trigonometric polynomial. Moreover, adding zero vectors $x_k$ or $y_k^*$ and enlarging $n$ if necessary, we can assume that $Q =  \sum_{|k|\leq n} e_{-k} y_k^*$.

To prove \eqref{eq:toprovebddTorus} observe that for $E:=\Ell^{\min(p,q')}(\Rd)$ and $f\in E\otimes X$, $g\in E\otimes Y^{*}$, the boundedness of $T_m$ is equivalent to
\begin{equation}\label{eq:bddTmequiv}
\Big|\int_{\R^d} \lb m(\xi) \wh f(\xi), \wh g(\xi)\rb \, \ud \xi\Big|\leq \|T_m\| \, \|f\|_{\Ell^p(\R^d;X)} \|g\|_{\Ell^{q'}(\R^d;Y^*)},
\end{equation}
where we have used that $\lb m \wh f, \wh g\rb = \lb T_m f, g\rb$. Let $h(t) := \F^{-1}(\one_{[0,1]^d})(t) = e^{i\pi (t_1+\ldots +t_d)} \prod_{j=1}^d \frac{\sin(\pi t_j)}{\pi t_j}$ for $t=(t_{1},\ldots, t_{n})\in\Rd$, and
\[
f(t): = a^{d/p} h(at) P(at),\quad g(t) := a^{d/q} h(at) Q(-at).
\]
Then $f\in E\otimes X$, $g\in E\otimes Y^{*}$, and
\[
\wh f(\xi) = a^{-d/p'} \sum_{|k|\leq n} \one_{[0,a]^d + ak}(\xi)  x_k,\quad\wh g(\xi) = a^{-d/q} \sum_{|k|\leq n} \one_{[0,a]^d +ak}(\xi)  y_k^*
\]
for $\xi\in\Rd$. By substitution we find
\begin{align*}
\|f\|_{\Ell^{p}(\R^d;X)} &= \Big(\int_{\R^d} |h(t)|^p \|P(t)\|^p_X \,\ud t\Big)^{1/p}
 = \Big(\sum_{j\in \Z^d}\int_{[0,1]^d+j} |h(t)|^p \|P(t)\|^p_X \,\ud t\Big)^{1/p}
\\ & = \Big(\int_{[0,1]^d} |H(t)|^p \|P(t)\|^p_X \,\ud t\Big)^{1/p}
\leq C_{d,p} \|P\|_{\Ell^p(\T^d;X)},
\end{align*}
where we used the standard fact that $H(t) = \sum_{j\in \Z^d}|h(t+j)|^p\leq C_{d,p}$ for $t\in\Rd$, $p\in (1, \infty)$ and some $C_{d,p}\geq 0$. Similarly, one checks that
\[
\|g\|_{\Ell^{q'}(\R^d;Y^*)} \leq C_{d,q'} \|Q\|_{\Ell^{q'}(\T^d;Y^*)}.
\]
Since the left-hand side of \eqref{eq:bddTmequiv} equals the left-hand side of \eqref{eq:toprovebddTorus}, the first statement follows from these estimates.

The second statement follows from the first since the $X$-valued trigonometric polynomials are dense in $\Ell^{p}(\T^{d};X)$.
\end{proof}

\begin{remark}\label{rem:Torusp=q}
Any Fourier multiplier from $\Ellp(\T^d;X)$ to $\Ellq(\T^d;Y)$ with $1\leq p\leq q\leq \infty$, trivially yields a multiplier from $\Ell^u(\T^d;X)$ into $\Ell^v(\T^d;Y)$ for all $p\leq u\leq v\leq q$. Indeed, this follows from the embedding $\Ell^a(\T^d;X)\hookrightarrow \Ell^b(\T^d;X)$ for $a\geq b$. In particular, any boundedness result from $\Ell^p(\T^d;X)$ to $\Ell^q(\T^d;Y)$ implies boundedness from $\Ell^u(\T^d;X)$ into $\Ell^u(\T^d;Y)$.
\end{remark}

As an application of Proposition \ref{prop:transference} and Theorem \ref{thm:sharpintegrabilityintro} we obtain the following:

\begin{corollary}\label{co:sharpintegrabilityintroTorus}
Let $X$ be a Banach space with type $p_0\in(1,2]$ and $Y$ a Banach space with cotype $q_0\in[2,\infty)$, and let $p\in (1, p_0)$, $q\in (q_0, \infty)$. Let $r\in(1,\infty]$ be such that $\frac{1}{r}=\frac{1}{p}-\frac{1}{q}$. If $(m_k)_{k\in \Z^d}$ is a family of operators in $\La(X,Y)$ and
\begin{align*}
\{(\abs{k}^{d/r}+1)m_k\mid k\in \Z^{d}\}\subseteq\La(X,Y)
\end{align*}
is $\gamma$-bounded, then the Fourier multiplier operator with symbol $(m_{k})_{k\in\Z}$ is bounded from $\Ellp(\T^{d};X)$ to $\Ellq(\T^{d};Y)$.
Moreover, if $p_0 = 2$ (or $q_0 = 2)$, then one can also take $p = 2$ (or $q=2$).
\end{corollary}
\begin{proof}
Let $m(\xi) := \sum_{k\in \Z^d} \one_{[0,1]^d}(\xi-k) m_k$ for $\xi\in \R^d$. Then for $k\in\Z^{d}$ and $\xi\in [0,1]^d+k$, we have $m(\xi) = m_k$ and $|\xi|^{d/r} \leq (|k|+\sqrt{d})^{d/r}\leq C_{d,r}(|k|^{d/r}+1)$. Therefore, Kahane's contraction principle yields
\[
\gamma(\{\abs{\xi}^{\frac{d}{r}}m(\xi)\mid \xi\in\Rd\}) \leq C_{d,r}\gamma(\{(\abs{k}+1)^{\frac{d}{r}}m_k\mid k\in \Z^d\}),
\]
which is assumed to be finite.  By Theorem \ref{thm:sharpintegrabilityintro}, $T_m:\Ell^p(\R^d;X)\to \Ell^q(\R^d;Y)$ is bounded. Since
$m_k = \int_{[0,1]^d} m(t+k) \,\ud t$ for $k\in\Z^{d}$, Proposition \ref{prop:transference} yields the required result.
\end{proof}

As an application we show how Corollary \ref{co:sharpintegrabilityintroTorus} can be used in the study of Schur multipliers. For $p\in[1, \infty)$ let $\schatten^p$ denote the Schatten $p$-class over a Hilbert space $H$. For a detailed discussion on these spaces we refer to \cite{DiJaTo95} and \cite{HyNeVeWe16}. Let $(e_{j})_{j\in \Z}$ be a countable spectral resolution of $H$. That is,
\begin{enumerate}[(1)]
\item for all $j\in \Z$, $e_{j}$ is an orthogonal projection in $H$;
\item for all $j,k\in \Z$, $e_{j} e_{k} = 0$ if $j\neq k$;
\item for all $h\in H$, $\sum_{j\in \Z} e_j h = h$.
\end{enumerate}

Using the technique of \cite[Theorem 4]{Potapov-Sukochev11} we deduce the following result from Corollary \ref{co:sharpintegrabilityintroTorus}. A similar result holds for more general noncommutative $\Ell^p$-spaces with a similar proof.

\begin{corollary}\label{cor:Schattenappl}
Let $a\in (1, \infty)\setminus\{2\}$ and let $r\in [1,\infty)$ be such that $\frac1r < |\frac{1}{a} - \frac12|$. Let $m:\Z\to\C$ be such that $C_m:=\sup_{j\in \Z} (1+|j|^{1/r}) |m_j|<\infty$, let $f:\Z\to \Z$ and write $m_{j,k}^f := m_{f(j)-f(k)}$. Then the Schur multiplier operator $M_{m,f}^e$ on $\Cs^a$, given by
\begin{align}\label{eq:DefMve}
M_{m,f}^e v := \sum_{j,k\in \Z} m_{j,k}^f e_j v e_k = \lim_{n\to \infty} \sum_{|j|,|k|\leq n} m_{j,k}^f e_j v e_k
\end{align}
for $v\in \Cs^{a}$, is well-defined and satisfies
\begin{equation}\label{eq:Mmeafschatting}
\|M_{m,f}^e\|_{\calL(\Cs^a)}\leq C_{a,r} C_m
\end{equation}
for some $C_{a,r}\geq 0$ independent of $m$.
\end{corollary}
\begin{proof}
By duality it suffices to consider $a\in (1, 2)$, and by an approximation argument it suffices to consider finite rank operators $v\in \Cs^a$.
Let $p\in (1, a)$ be such that $\frac{1}{p} - \frac12 = \frac1r$. Since $\Cs^a$ has type $a$ and cotype $2$ (see \cite{HyNeVeWe2}) it follows from Theorem \ref{co:sharpintegrabilityintroTorus} that the Fourier multiplier $T_m$ associated with $(m_n)_{n\in \Z}$ is bounded from $L^a(\T;\Cs^a)$ to $L^2(\T;\Cs^a)$ with
\begin{equation}\label{eq:TmbddSchur}
\|T_m\|_{\calL(L^a(\T;\Cs^a), L^2(\T;\Cs^a))}\leq C_{p,a} C_m.
\end{equation}

As in the proof of \cite[Theorem 4]{Potapov-Sukochev11} one sees that
\begin{align*}
\|M_{m,f}^e v\|_{\Cs^a} &=
\Big\|\sum_{n\in \Z} m_n e^{2\pi i nt} v_n\Big\|_{\Cs^a} = \|T_m ((v_n)_{n\in \Z})(t)\|_{\Cs^a},
\end{align*}
where $v_n := \sum_{j,k\in \Z, f(j)-f(k) = n} e_j v e_k$ for $n\in \Z$. Similarly,
\[
\|v\|_{\Cs^a} = \Big\|\sum_{n\in \Z} e^{2\pi i nt} v_n\Big\|_{\Cs^a}.\]
Taking $L^q$ and $L^p$ norms over $t\in [0,1]$ in the above identities yields
\begin{align*}
\|M_{m,f}^e v\|_{\Cs^a} &= \|T_{m} (v_n)_{n\in \Z}\|_{L^q(0,1;\Cs^a)}
\\ & \leq \|T_m\|_{\calL(L^p(\T;\Cs^a),L^q(\T;\Cs^a))} \Big\|\sum_{n\in \Z} e^{2\pi i nt} v_n\Big\|_{L^p(0,1;\Cs^a)} = C_{p,a} C_m \|v\|_{\Cs^a},
\end{align*}
where we applied \eqref{eq:TmbddSchur} in the final step.
\end{proof}

\begin{problem}\label{prob:Schur}
Can we take $\frac1r = |\frac{1}{a} - \frac12|$ in Corollary \ref{cor:Schattenappl}?
\end{problem}

If the answer to the question in Problem \ref{prob:Schur} is negative, then the limitations of Theorem \ref{thm:sharpintegrabilityintro} and Corollary \ref{co:sharpintegrabilityintroTorus} are natural. Moreover, from the proof of the latter (see Theorem \ref{thm:sharpintegrability} below) it would then follow that the embedding $H^{\frac1a-\frac12}_a(\R;\Cs^a)\to \gamma(\R;\schatten^a)$ does not hold for $a\in (1,2)$. Here $\gamma(\R;\Cs^{a})$ is the $\Cs^{a}$-valued $\gamma$-space used in the proof of Theorem \ref{thm:sharpintegrability}.

\subsection{Fourier type assumptions}\label{Fourier type assumptions}

Before turning to more advanced multiplier theorems, we start with the case where we use the Fourier type of the Banach spaces to derive an analogue of the basic estimate $\|T_m\|_{\calL(\Ell^2(\R^d))}\leq \|m\|_{\infty}$.

\begin{proposition}\label{Lp-Lq multipliers Fourier type}
Let $X$ be a Banach space with Fourier type $p\in[1,2]$ and $Y$ a Banach space with Fourier cotype $q \in[2,\infty]$, and let $r\in[1,\infty]$ be such that $\tfrac{1}{r}=\tfrac{1}{p}-\tfrac{1}{q}$. Let $m:\Rd\to\La(X,Y)$ be an $X$-strongly measurable map such that $\norm{m(\cdot)}_{\La(X,Y)}\in\Ellr(\Rd)$. Then $T_{m}$ extends uniquely to a bounded map from $\Ellp(\Rd;X)$ into $\Ellq(\Rd;Y)$ with
\begin{align*}
\|T_{m}\|_{\La(\Ellp(\Rd;X),\Ellq(\Rd;Y))}\leq \F_{p,X,d}\,\F_{q',Y,d}\norm{\|m(\cdot)\|_{\La(X,Y)}}_{\Ellr(\Rd)}.
\end{align*}
\end{proposition}

In Proposition \ref{prop:converseFouriertype3} we show that this multiplier result characterizes the Fourier type $p$ of $X$ for specific choices of $Y$, and the Fourier cotype $q$ of $Y$ for specific choices of $X$.
\begin{proof}
Let $f\in\Sw(\Rd;X)$. By H\"{o}lder's inequality,
\begin{align*}
\|m\widehat{f}\|_{\Ellqprime(\Rd;Y)}&\leq \norm{\|m(\cdot)\|_{\La(X,Y)}}_{\Ellr(\Rd)}\|\widehat{f}\|_{\Ellpprime(\Rd;X)}\\
&\leq \F_{p,X,d}\norm{\|m(\cdot)\|_{\La(X,Y)}}_{\Ellr(\Rd)}\|f\|_{\Ellp(\Rd;X)}.
\end{align*}
Since $\|\F^{-1}(g)\|_{\Ellq(\Rd;Y)}=\|\F(g)\|_{\Ellq(\Rd;Y)}$ for $g\in\Ellqprime(\Rd;Y)$, it follows that
\begin{align*}
\|T_{m}(f)\|_{\Ellq(\Rd;Y)}&\leq \F_{q',Y,d}\|m\widehat{f}\|_{\Ellqprime(\Rd;Y)}\\
&\leq \F_{p,X,d}\,\F_{q',Y,d}\norm{\|m(\cdot)\|_{\La(X,Y)}}_{\Ellr(\Rd)}\|f\|_{\Ellp(\Rd;X)},
\end{align*}
which concludes the proof.
\end{proof}

\begin{remark}\label{rem:Young's inequality}
It follows from Young's inequality (see \cite[Exercise 4.5.4]{Grafakos08} or \cite[Proposition 1.3.5]{ArBaHiNe11}) that $T_{m}:\Ellp(\Rd;X)\to\Ellq(\Rd;Y)$ is bounded with
\begin{align}\label{eq:young's inequality Fourier type}
\|T_{m}\|_{\La(\Ellp(\Rd;X),\Ellq(\Rd;Y))}\leq \|\F^{-1}m\|_{\Ell^{r'}(\Rd;\La(X,Y))}
\end{align}
for all $X$ and $Y$, $1\leq p\leq q\leq \infty$ and $r\in[1,\infty]$ such that $\tfrac{1}{r}=\tfrac{1}{p}-\tfrac{1}{q}$, and all $X$-measurable $m:\Rd\to\La(X,Y)$ of moderate growth at infinity for which $\F^{-1}m\in\Ell^{r'}(\Rd;\La(X,Y))$. In certain cases \eqref{eq:young's inequality Fourier type} is stronger than the result in Proposition \ref{Lp-Lq multipliers Fourier type}. For instance, if $r\in[1,2]$ and $\La(X,Y)$ has Fourier type $r$ (for $r>1$ this implies that either $X$ or $Y$ is finite-dimensional),
then
\begin{align*}
\|T_{m}\|_{\La(\Ellp(\Rd;X),\Ellq(\Rd;Y))}\leq \|\F^{-1}m\|_{\Ell^{r'}(\Rd;\La(X,Y))}\leq C\|m\|_{\Ell^{r}(\Rd;\La(X,Y))}
\end{align*}
for some constant $C\geq 0$. Therefore we recover the conclusion of Proposition \ref{Lp-Lq multipliers Fourier type} from Young's inequality in a very special case.
\end{remark}

\begin{remark}\label{rem:more general groups}
Proposition \ref{Lp-Lq multipliers Fourier type} (and Theorem \ref{Lp-Lq multipliers Fourier type2} below) can also be formulated for general abelian locally compact groups $G$, not just for $\Rd$. In that case one should assume that the Fourier transform is bounded from $\Ell^{p}(G;X)$ to $\Ell^{p'}(\widehat{G};X)$ for $p\in[1,2]$ and that the inverse Fourier transform is bounded from $\Ell^{q'}(\widehat{G};Y)$ to $\Ellq(G;Y)$ for $q\in[2,\infty]$. Here $\widehat{G}$ is the dual group of $G$. Then one works with symbols $m:\widehat{G}\to\La(X,Y)$ which are $X$-strongly measurable and such that $[\xi\mapsto\norm{m(\xi)}_{\La(X,Y)}]\in\Ellr(\widehat{G})$, where $\tfrac{1}{r}=\tfrac{1}{p}-\tfrac{1}{q}$. In the same way as in Proposition \ref{Lp-Lq multipliers Fourier type}, one then obtains a constant $C\geq 0$ independent of $m$ such that
\begin{align*}
\|T_{m}\|\leq C\norm{\|m(\cdot)\|_{\La(X,Y)}}_{\Ellr(\widehat{G})}.
\end{align*}
For $G = \T^d$ such results can also be deduced from the $\R^d$-case by applying the transference of Proposition \ref{prop:transference}.
\end{remark}

In the scalar setting we noted in \eqref{weakLrcondition} that the conclusion of Proposition \ref{Lp-Lq multipliers Fourier type} holds under the weaker condition $m\in \Ell^{r,\infty}(\R^d)$. In certain cases we can prove such a result in the vector-valued setting.

\begin{theorem}\label{Lp-Lq multipliers Fourier type2}
Let $X$ be a Banach space with Fourier type $p_0\in(1,2]$ and $Y$ a Banach space with Fourier cotype $q_0\in[2,\infty)$, and let $p\in (1,p_0)$ and $q\in (q_0, \infty)$. Let $r\in[1,\infty]$ be such that $\tfrac{1}{r}=\tfrac{1}{p}-\tfrac{1}{q}$. Let $m:\Rd\to\La(X,Y)$ be an $X$-strongly measurable map such that $[\xi\mapsto\norm{m(\xi)}_{\La(X,Y)}]\in\Ell^{r,\infty}(\Rd)$. Then $T_{m}$ extends uniquely to a bounded map from $\Ellp(\Rd;X)$ into $\Ellq(\Rd;Y)$ with
\begin{align*}
\|T_{m}\|_{\La(\Ellp(\Rd;X),\Ellq(\Rd;Y))}\leq C\norm{\|m(\cdot)\|_{\La(X,Y)}}_{\Ell^{r,\infty}(\Rd)},
\end{align*}
where $C\geq 0$ is independent of $m$.
\end{theorem}
\begin{proof}
Observe that by real interpolation (see \cite[1.18.6]{Triebel95} and \cite[(2.33)]{KrPeSe82}) we obtain $\F:\Ell^{v', \infty}(\R^d;Y)\to \Ell^{v, \infty}(\R^d;Y)$ for all $v\in (q_0,\infty)$.

Let $p_1, p_2, q_1, q_2\in (1, \infty)$ be such that
\[\frac{1}{p_1} = \frac{1}{p}+\varepsilon, \ \frac{1}{p_2} = \frac{1}{p} - \varepsilon, \ \frac{1}{q_1} = \frac{1}{q}+\varepsilon, \ \frac{1}{q_2} = \frac{1}{q} - \varepsilon\]
for $\varepsilon>0$ so small that $p_2<p_0$ and $q_1>q_0$. Note that
\[
\frac{1}{p_j} - \frac{1}{q_j} = \frac{1}{p} - \frac{1}{q} = \frac1r.
\]
Let $f\in\Sw(\Rd;X)$. By H\"{o}lder's inequality (see \cite[Exercise 1.4.19]{Grafakos08} or \cite[Theorem 3.5]{O'Neil63}), for $j=1,2$,
\begin{align*}
\|m\widehat{f}\|_{\Ell^{q_j',\infty}(\Rd;Y)}&\leq C \norm{\|m(\cdot)\|_{\La(X,Y)}}_{\Ell^{r,\infty}(\Rd)}\|\widehat{f}\|_{\Ell^{p_j'}(\Rd;X)}\\
&\leq C\norm{\|m(\cdot)\|_{\La(X,Y)}}_{\Ell^{r,\infty}(\Rd)}\|f\|_{\Ell^{p_j}(\Rd;X)}
\end{align*}
for $C\geq 0$ independent of $m$ and $f$, where we used the Fourier type $p_j$ of $X$ and $\|\cdot\|_{p_j',\infty}\leq \|\cdot\|_{p_j'}$. It follows from the first observation and the estimate above that
\begin{align*}
\|T_{m}(f)\|_{\Ell^{q_j, \infty}(\Rd;Y)}&\leq C\|m\widehat{f}\|_{\Ell^{q_j',\infty}(\Rd;Y)}\\
&\leq C\norm{\|m(\cdot)\|_{\La(X,Y)}}_{\Ell^{r,\infty}(\Rd)}\|f\|_{\Ell^{p_j}(\Rd;X)}.
\end{align*}
Hence $T_m:\Ell^{p_j}(\Rd;X)\to \Ell^{q_j, \infty}(\Rd;Y)$ is bounded for $j\in \{1,2\}$.
By real interpolation (see \cite[Theorem 1.18.6.2]{Triebel95}) we find that $T_m:\Ellp(\R^d;X)\to \Ell^{q,p}(\R^d;Y)$ and the required result follows from $\Ell^{q,p}(\R^d;Y)\hookrightarrow \Ellq(\R^d;Y)$ (see \cite[Proposition 1.4.10]{Grafakos08}).
\end{proof}

The above result provides an analogue of \cite[Theorem 1.12]{Hormander60}. In general, we do not know the ``right'' geometric conditions under which such a result holds. We formulate the latter as an open problem.

\begin{problem}
Let $1<p\leq 2\leq q<\infty$ and let $r\in[1,\infty]$ be such that $\tfrac{1}{r}=\tfrac{1}{p}-\tfrac{1}{q}$. Classify those Banach spaces $X$ and $Y$ for which $T_{m}\in\La(\Ellp(\Rd;X),\Ellq(\Rd;Y))$ for all $X$-strongly measurable maps $m:\Rd\to\La(X,Y)$ such that $\|m(\cdot)\|_{\calL(X,Y)}\in \Ell^{r,\infty}(\R^d)$.
\end{problem}
A similar question can be asked for the case where $X = Y$ and $m$ is scalar-valued.

We will now show that the Fourier multiplier result in Proposition \ref{Lp-Lq multipliers Fourier type} characterizes the Fourier type of the underlying Banach spaces. To this end we need the following lemma.

\begin{lemma}\label{lem:estimate for sharpness}
Let $X$ and $Y$ be Banach spaces. Let $p\in [1,2]$, $q\in [2,\infty]$ and $r\in (1, \infty]$ be such that $\tfrac{1}{r} = \tfrac{1}{p} - \tfrac1{q}$. Assume that for all $m\in \Ellr(\Rd;\La(X,Y))$ the operator $T_m:\Ellp(\Rd;X)\to \Ellq(\Rd;Y)$ is bounded. Then there is a constant $C\geq0$ such that
for all $f\in \Sw(\Rd;X)$ and $g\in \Sw(\R^d;Y^*)$
\begin{equation}\label{eq:toprovebil}
\big\| \|\widehat{f}(\cdot)\|_X \|\widehat{g}(\cdot)\|_{Y^*} \big\|_{\Ell^{r'}(\Rd)} \leq C\|f\|_{\Ellp(\Rd;X)} \|g\|_{\Ell^{q'}(\Rd;Y^*)}.
\end{equation}
\end{lemma}
\begin{proof}
By the closed graph theorem there exists a constant $C\geq0$ such that
\begin{align*}
|\lb T_m f, g\rb| \leq C \|m\|_{\Ellr(\Rd;\La(X,Y))} \|f\|_{\Ellp(\Rd;X)} \|g\|_{\Ell^{q'}(\Rd;Y^*)}
\end{align*}
for all $f\in \Ellp(\Rd;X)$, $g\in\Ell^{q'}(\Rd;Y^{*})$ and $m\in \Ell^r(\Rd;\La(X,Y))$. It follows that, for all $f\in \Sw(\Rd;X)$ with $\|f\|_{p}\leq 1$ and $g\in \Sw(\Rd;Y^*)$ with $\|g\|_{q'}\leq 1$,
\begin{align}\label{eq:helpequation}
|\lb m \widehat{f}, \widehat{g}\rb| = |\lb T_m f, g\rb| \leq C \|m\|_{\Ellr(\Rd;\La(X,Y))}.
\end{align}
It suffices to show \eqref{eq:toprovebil} for fixed $f\in \Sw(\Rd;X)$ with $\|f\|_{p}=1$ and $g\in \Sw(\Rd;Y^*)$ with $\|g\|_{q'}=1$. Let $\varepsilon\in (0,1)$ and choose simple functions $\zeta:\R^d\to X$ and $\eta:\R^d\to Y^*$ such that
$\|\zeta-\widehat{f}\|_{p'}\leq \min(\varepsilon^{\frac12},\varepsilon \|\widehat{g}\|_{q}^{-1})$ and $\|\eta-\widehat{g}\|_{q}\leq \min(\varepsilon^{\frac12},\varepsilon\|\widehat{f}\|_{p'}^{-1})$. Then, by H\"older's inequality with $\frac1r + \frac{1}{p'} + \frac{1}{q}=1$ and by \eqref{eq:helpequation}, it follows that
\begin{align}
\nonumber |\lb m \zeta, \eta\rb| &\leq |\lb m (\zeta-\widehat{f}), \eta - \widehat{g}\rb| +|\lb m (\zeta-\widehat{f}),\widehat{g}\rb|+|\lb m\widehat{f}, \eta - \widehat{g}\rb|+ |\lb m \widehat{f}, \widehat{g}\rb|
\\ & \label{eq:nowitworks} \leq \|m\|_r \Big(\|\zeta-\widehat{f}\|_{p'} \|\eta-\widehat{g}\|_{q} + \|\zeta-\widehat{f}\|_{p'}\|\widehat{g}\|_{q}+\|\widehat{f}\|_{p'}\|\eta-\widehat{g}\|_{q}+ C\Big) \\
\nonumber
&\leq \|m\|_r (3\varepsilon + C)
\end{align}
for all $m\in \Ellr(\Rd;\La(X,Y))$. By considering a common refinement, we may suppose that $\zeta = \sum_{k=1}^n \one_{A_k} x_k$ and $\eta = \sum_{k=1}^n \one_{A_k} y_k^*$ for $n\in\N$, $x_{1}, \ldots, x_{n}\in X$, $y_{1}^{*},\ldots, y_{n}^{*}\in Y^{*}$ and $A_{1},\ldots, A_{n}\subseteq\Rd$ disjoint and of finite measure $\abs{A_{k}}$. For $1\leq k\leq n$ let $x_k^*\in X^*$ and $y_k\in Y$ of norm one be such that $\lb x_k,x_k^*\rb = \|x_k\|$ and $\lb y_k, y_k^*\rb\geq (1-\varepsilon) \|y_k^*\|$. Let $m:\R^d\to \calL(X,Y)$ be given by
\begin{align*}
m(\xi) x :=\sum_{k=1}^n c_k \one_{A_k}(\xi) \lb x, x_k^*\rb y_k\qquad(\xi\in\Rd, x\in X),
\end{align*}
where $c_1, \ldots, c_n\in\R$. Then \eqref{eq:nowitworks} implies
\begin{align*}
(1-\varepsilon) \sum_{k=1}^n c_{k}\abs{A_k} \|x_k\| \, \|y^{*}_k\| \leq  (C + 3\varepsilon) \Big(\sum_{k=1}^n |c_k|^r \abs{A_k}\Big)^{\frac1r}.
\end{align*}
By taking the supremum over all $c_k$'s with $\sum_{k=1}^n |c_k|^r \abs{A_k}\leq 1$ we find
\begin{align*}
(1-\varepsilon)\big\| \|\zeta(\cdot)\|_X \|\eta(\cdot)\|_{Y^*} \big\|_{\Ell^{r'}(\Rd)}=(1-\varepsilon)  \Big(\sum_{k=1}^n \abs{A_k} \|x_k\|^{r'} \|y^{*}_k\|^{r'}\Big)^{\frac{1}{r'}} \leq  (C +3\varepsilon).
\end{align*}
Therefore, using this estimate, the reverse triangle inequality and H\"{o}lder's inequality (with $\tfrac{1}{r'}=\tfrac{1}{p'}+\tfrac{1}{q}$), we obtain
\begin{align*}
\big\| &\|\widehat{f}(\cdot)\|_X \|\widehat{g}(\cdot)\|_{Y^*} \big\|_{\Ell^{r'}(\Rd)}
\\ & \leq \big\| \|\widehat{f}(\cdot)\|_X \|\widehat{g}(\cdot)\|_{Y^*}-\|\zeta(\cdot)\|_{X}\|\widehat{g}(\cdot)\|_{Y^{*}} \big\|_{\Ell^{r'}(\Rd)}\\
&\quad +\big\| \|\zeta(\cdot)\|_{X}\|\widehat{g}(\cdot)\|_{Y^{*}}-\|\zeta(\cdot)\|_X \|\eta(\cdot)\|_{Y^*} \big\|_{\Ell^{r'}(\Rd)} +\big\| \|\zeta(\cdot)\|_X \|\eta(\cdot)\|_{Y^*} \big\|_{\Ell^{r'}(\Rd)}
\\ & \leq
\big\| \|\widehat{f}(\cdot)-\zeta(\cdot)\|_X \|\widehat{g}(\cdot)\|_{Y^*}\big\|_{\Ell^{r'}(\Rd)}
+ \big\| \|\zeta(\cdot)\|_X \|\eta(\cdot)-\widehat{g}(\cdot)\|_{Y^*}\big\|_{\Ell^{r'}(\Rd)}
+\frac{C +3\varepsilon}{1-\varepsilon}
\\ & \leq  \|\widehat{f}-\zeta\|_{p'}\|\widehat{g}\|_{q} +  \|\zeta\|_{p'}\|\eta-\widehat{g}\|_{q} + \frac{C +3\varepsilon}{1-\varepsilon}
\\ & \leq \varepsilon + (\|\widehat{f}-\zeta\|_{p'}+\|\widehat{f}\|_{p'})\|\eta-\widehat{g}\|_{q}
+\frac{C +3\varepsilon}{1-\varepsilon}
\leq
3\varepsilon + \frac{C +3\varepsilon}{1-\varepsilon}.
\end{align*}
Letting $\epsilon$ tend to zero yields \eqref{eq:toprovebil} for $\|f\|_{p}=1=\|g\|_{q'}$, as was to be shown.
\end{proof}

Now we are ready to show that, by letting $Y$ vary, the Fourier multiplier result in Proposition \ref{Lp-Lq multipliers Fourier type} characterizes the Fourier type of $X$, and vice versa.

\begin{proposition}\label{prop:converseFouriertype3}
Let $X$ and $Y$ be Banach spaces. Let $\frac{1}{r} = \frac{1}{p} - \frac{1}{q}$ with $p\in [1, 2]$, $q\in [2, \infty)$ and $r\in (1, \infty]$. Assume that for all $m\in \Ellr(\Rd;\La(X,Y))$ the operator $T_m:\Ell^p(\Rd;X)\to \Ell^q(\Rd;Y)$ is bounded.
\begin{enumerate}[$(1)$]
\item\label{item:converseFouriertype1} If $Y=\C$ and $q=2$, then $X$ has Fourier type $p$.
\item\label{item:converseFouriertype2} If $X = \C$ and $p=2$, then $Y$ has Fourier type $q'$.
\item\label{item:converseFouriertype3} If $Y = X^*$ and $q = p'$, then $X$ has Fourier type $p$.
\end{enumerate}
\end{proposition}
\begin{proof}
By Lemma \ref{lem:estimate for sharpness}, \eqref{eq:toprovebil} holds for some $C\geq 0$. Therefore in case \eqref{item:converseFouriertype1} we obtain, for fixed $f\in \Sw(\Rd;X)$ and for all $\varphi\in \Sw(\Rd)$,
\begin{align*}
\big\| \|\widehat{f}(\cdot)\|_X \abs{\varphi(\cdot)}  \big\|_{\Ell^{r'}(\Rd)} \leq C \|f\|_{\Ellp(\Rd;X)} \|\varphi\|_{\Elltwo(\Rd)},
\end{align*}
where we used the fact that $\F:\Elltwo(\Rd)\to \Elltwo(\Rd)$ is an isometry. Taking the supremum over all $\|\varphi\|_{\Elltwo(\Rd)}\leq 1$ we see that
\begin{align*}
\|\widehat{f}\|_{\Ell^{p'}(\Rd;X)}\leq C\|f\|_{\Ellp(\Rd;X)},
\end{align*}
and hence $X$ has Fourier type $p$. In case \eqref{item:converseFouriertype2} we deduce in the same way that $Y^*$ has Fourier type $q'$ and thus also that $Y$ has Fourier type $q'$, by duality.

Finally, for \eqref{item:converseFouriertype3} note that $\frac{1}{r'} = \frac{2}{p'}$. Thus, taking $f = g\in \Sw(\Rd;X)$ in \eqref{eq:toprovebil} yields
\[\|\widehat{f}\|_{\Ell^{p'}(\R^d;X)}^2 \leq C\|f\|_{\Ell^p(\R^d;X)}^2,\]
and the result follows.
\end{proof}

\begin{remark}
An alternative proof of Proposition \ref{prop:converseFouriertype3} can be given using the transference of Proposition \ref{prop:transference}. However, this yields worse bounds and it seems that the analogue in the type-cotype setting requires the same technique as in Proposition \ref{prop:converseFouriertype3}. The estimate which can be proved under the assumption of Lemma \ref{lem:estimate for sharpness} is as follows.
There is a constant $C\geq0$ such that
for all $(x_k)_{|k|\leq n}$ in $X$ and $(y_k^*)_{|k|\leq n}$ in $Y^*$,
\[
\Big(\sum_{|k|\leq n} \|x_k\|_X^{r'} \|y_k^*\|_Y^{r'}\Big)^{\frac{1}{r'}} \leq C\Big\|\sum_{|k|\leq n} e_k x_k\Big\|_{\Ell^p(\T^d;X)} \Big\|\sum_{|k|\leq n} e_k y_k^*\Big\|_{\Ell^{q'}(\T^d;Y^*)}.
\]
\end{remark}

We end this section with a simple example which shows that the geometric limitation in Theorem \ref{Lp-Lq multipliers Fourier type} is also natural in the case $X = Y = \ell^u$. We will come back to this in Example \ref{example sharpness R-bounds}, where type and cotype will be used to derive different results.

\begin{example}\label{ex:Fouriertype}
Let $p\in (1, 2]$, and for $q\in [2, \infty)$ let $r\in(1,\infty]$ be such that $\frac{1}{r} = \frac{1}{p} - \frac{1}{q}$.
Let $u\in[1, \infty)$ and let $X:=\ell^{u}$. Let $(e_j)_{j\in\N_{0}}\subseteq X$ be the standard basis of $X$, and for $k\in\N$ let $S_{k}\in \La(X)$ be such that $S_{k}(e_{j}):= e_{j+k}$ for $j\in\N_{0}$. Let $m:\R\to\La(\ell^{u})$ be given by $m(\xi):=\sum_{k=1}^{\infty} c_k \ind_{(k-1,k]}(\xi)S_{k}$ for $\xi\in\R$, where $c_k = k^{-\frac1r} \log(k+1)^{-2}$ for $k\in\N$. Observe that
\begin{align*}
\int_{\R} \|m(\xi)\|_{\calL(X)}^r \,\ud \xi = \sum_{k=1}^{\infty} c_k^r <\infty,
\end{align*}
with the obvious modification for $r=\infty$.
If $u \in [p,p']$, then $X$ has Fourier type $p$ and Fourier cotype $q=p'$. Thus by Proposition \ref{Lp-Lq multipliers Fourier type}, in this case $T_m:\Ellp(\R;X)\to\Ellq(\R;X)$ is bounded.

We show that this result is sharp in the sense that for $u\notin [p,p']$ the conclusion is false. This shows that Proposition \ref{Lp-Lq multipliers Fourier type} is optimal in the exponent of the Fourier type of the space for $X=Y=\ell^{u}$.

Let $q\in [2, \infty)$ and assume that $T_m\in \calL(\Ellp(\R;X),\Ellq(\R;X))$. Let, for $k\in\N$, $\ph_{k}:\R\to\C$ be such that $\widehat{\ph_{k}}=\one_{(k-1,k]}$ and let, for $n\in\N$, $f := \sum_{k=n+1}^{2n} \ph_{k} e_0$. Then
\begin{align*}
\|T_{m}(f)(t)\|_{X} =  \Big\|\sum_{k=n+1}^{2n} c_k \varphi_{k}(t) e_{k}\Big\|_{\ell^u} = \Big(\sum_{k=n+1}^{2n} |c_k|^u |\varphi_{k}(t)|^u\Big)^{\frac1u}
\end{align*}
for each $t\in\R$. Since $|\varphi_{k}(t)| = \big|\frac{\sin(\pi t)}{\pi t}\big|$ for all $t\in\R$ and $k\in\N_{0}$,
\begin{align*}
\|T_{m}(f)\|_{\Ell^q(\R;X)} \geq n^{\frac1u} |c_{2n}|  \|\varphi_1\|_{\Ell^q(\R)} \geq C_1 n^{\frac1u - \frac1r}  \log(n)^{-2}.
\end{align*}
for some $C_1\in(0,\infty)$. On the other hand, $\|f\|_{\Ell^p(\R;X)} = \big\|\sum_{k=n+1}^{2n} \varphi_k\big\|_{\Ell^p(\R)}$. Now, $\big|\sum_{k=n+1}^{2n} \varphi_{k}(t)\big| = \big|\frac{\sin(\pi nt)}{\pi t}\big|$ for all $t\in\R$, since $\sum_{k=n+1}^{2n} \widehat{\varphi}_{k} = \ind_{(n,2n]}$. Therefore there exists a constant $C_{2}\in(0,\infty)$ such that $\|f\|_{\Ell^p(\R;\ell^u)} = C_{2} n^{1-\frac{1}{p}}$.
It follows that
\begin{align*}
C_1 n^{\frac1u - \frac1r}  \log(n)^{-2} \leq \|T_m\|_{\calL(\Ellp(\R;X),\Ellq(\R;X))} C_{2} n^{1-\frac{1}{p}}.
\end{align*}
Letting $n\to \infty$ we deduce that $\frac{1}{u}\leq 1-\frac1p + \frac1r = \frac{1}{q'}$.  Thus, in the special case $q = p'$, we obtain $u\geq p$. By a duality argument one sees that also $u\leq p'$.
\end{example}

\subsection{Type and cotype assumptions}\label{sec:Bessel1}

In Proposition \ref{Lp-Lq multipliers Fourier type} and Theorem \ref{Lp-Lq multipliers Fourier type2} we obtained Fourier multiplier results under Fourier type assumptions on the spaces $X$ and $Y$. In this section we will present multiplier results under the less restrictive geometric assumptions of type $p$ and cotype $q$ on the underlying spaces $X$ and $Y$.

First we prove Theorem \ref{thm:sharpintegrabilityintro} from the Introduction.

\begin{theorem}\label{thm:sharpintegrability}
Let $X$ be a Banach space with type $p_0\in(1,2]$ and $Y$ a Banach space with cotype $q_0\in[2,\infty)$, and let $p\in (1, p_0)$ and $q\in (q_0, \infty)$, $r\in(1,\infty)$ be such that $\tfrac{1}{r}=\tfrac{1}{p}-\tfrac{1}{q}$. Let $m:\Rd\setminus\{0\}\to\La(X,Y)$ be an $X$-strongly measurable map such that
$\{\abs{\xi}^{\frac{d}{r}}m(\xi)\mid \xi\in\Rd\setminus\{0\}\}\subseteq\La(X,Y)$ is $\gamma$-bounded.
Then $T_{m}$ extends uniquely to a bounded map $\widetilde{T_{m}}\in\La(\Ellp(\Rd;X),\Ellq(\Rd;Y))$ with
\begin{align*}
\|\widetilde{T_{m}}\|_{\La(\Ellp(\Rd;X),\Ellq(\Rd;Y))}\leq C\gamma(\{\abs{\xi}^{\frac{d}{r}}m(\xi)\mid \xi\in\Rd\setminus\{0\}\}),
\end{align*}
where $C\geq0$ is independent of $m$. Moreover, if $p_0 = 2$ (or $q_0 = 2)$, then one can also take $p = 2$ (resp.~$q=2$).
\end{theorem}
It is unknown whether Theorem \ref{thm:sharpintegrability} holds with $p = p_0$ and $q = q_0$ (see Problem \ref{prob:pqequaltop0q0} below).

\begin{proof}
We will prove the result under the condition:
\begin{equation}\label{eq:embeddingcond}
\dot{\Hr}^{\frac{d}{p} - \frac{d}{2}}_{p}(\R^d;X) \hookrightarrow \gamma(\R^d;X) \ \ \text{and} \ \ \gamma(\R^d;Y) \hookrightarrow \dot{\Hr}^{\frac{d}{q} - \frac{d}{2}}_{q}(\R^d;Y).
\end{equation}
Here $\gamma(\Rd;X)$ is the $X$-valued $\gamma$-space (for more on these spaces see \cite{vanNeerven10}). Note that the assumptions imply \eqref{eq:embeddingcond}. Indeed, this follows from the homogeneous versions of \cite[Proposition 3.5]{Veraar13} and of \cite[Theorem 1.1]{KaNeVeWe08} (proved in exactly the same way, here we use the assumption that $X$ has type $p_{0}$ and $p<p_0$).
Moreover, if $p_0=2$, then $\dot{\Hr}^{0}_2(\R^d;X) = \Ell^2(\R^d;X)\hookrightarrow \gamma(\R^d;X)$ (see \cite[Theorem 11.6]{vanNeerven10}), hence in this case one can in fact take $p=2$.
The embedding for $Y$ follows in a similar way.

Let $m_1(\xi) := |\xi|^{\frac{d}{2} - \frac{d}{p}}$ and $m_2(\xi) := |\xi|^{\frac{d}{r}} m(\xi) m_1(\xi)$ for $\xi\in\Rd$. Let $f\in \Sw(\Rd;X)$. It follows from \eqref{eq:embeddingcond} that
\begin{align*}
\|T_{m}(f)\|_{\Ell^q(\Rd;Y)} & = \|T_{m_2} (f)\|_{\dot{\Hr}^{\frac{d}{q} - \frac{d}{2}}_q(\Rd;Y)}
\\ & \leq C\|T_{m_2}(f)\|_{\gamma(\Rd;Y)}\leq C_{1}\|m_2\widehat{f}\|_{\gamma(\Rd;Y)}\\
&\leq C\gamma(\{|\xi|^{\frac{d}{r}} m(\xi)\mid \xi\in\Rd\setminus\{0\}\})\|m_1 \widehat{f}\|_{\gamma(\Rd;X)}\\
&\leq  C\gamma(\{|\xi|^{\frac{d}{r}}m(\xi)\mid \xi\in\Rd\setminus\{0\}\})\|T_{m_1} f\|_{\gamma(\Rd;X)}\\
&\leq  C\gamma(\{|\xi|^{\frac{d}{r}} m(\xi)\mid \xi\in\Rd\setminus\{0\}\})\|T_{m_1} f\|_{\dot{\Hr}^{\frac{d}{p} - \frac{d}{2}}_p(\Rd;X)}\\
& =  C\gamma(\{|\xi|^{\frac{d}{r}}m(\xi)\mid \xi\in\Rd\setminus\{0\}\})\|f\|_{\Ellp(\Rd;X)},
\end{align*}
where we have used $\|f\|_{\gamma(\Rd;X)}=\|\widehat{f}\|_{\gamma(\Rd;X)}$ (see \cite{HyNeVeWe2}), the $\gamma$-multiplier Theorem (see \cite[Proposition 4.11]{Kalton-Weis04} and \cite[Theorem 5.2]{vanNeerven10}) and the fact that $\gamma(\Rd;X)=\gamma_{\infty}(\Rd;Y)$ because $Y$ does not contain a copy of $\textrm{c}_{0}$ (see \cite[Theorem 4.3]{vanNeerven10}). Since $\Sw(\Rd;X)\subseteq\Ellp(\Rd;X)$ is dense, this concludes the proof.
\end{proof}

In Theorem \ref{multiplier Bessel spaces} we provide conditions under which one can take $p = p_0$ and $q = q_0$.
The general case we state as an open problem:

\begin{problem}\label{prob:pqequaltop0q0}
Let $1\leq p\leq 2\leq q\leq \infty$ and $r\in(1,\infty]$ be such that $\tfrac{1}{r}=\tfrac{1}{p}-\tfrac{1}{q}$. Classify those Banach spaces $X$ and $Y$ for which $T_{m}\in\La(\Ellp(\Rd;X),\Ellq(\Rd;Y))$ for all $X$-strongly measurable maps $m:\Rd\to\La(X,Y)$ such that $\{|\xi|^{d/r} m(\xi):\xi\in \R^d\setminus\{0\}\}$ is $\gamma$-bounded.
\end{problem}

The same problem can be formulated in case $m$ is scalar-valued, in which case the $\gamma$-boundedness reduces to uniform boundedness.

\begin{remark}\label{property alpha}
Assume $X$ and $Y$ have property $(\alpha)$ as introduced in \cite{Pisier78}. (This implies that $X$ has finite cotype, and if $X$ and $Y$ are Banach lattices then property $(\alpha)$ is in fact equivalent to finite cotype.) In the multiplier theorems in this paper where $\gamma$-boundedness is an assumption, one can deduce a certain $\gamma$-boundedness result for the Fourier multiplier operators as well. Indeed, assume for example the conditions of Theorem \ref{thm:sharpintegrability}. Let $\{m_j:\R^d\setminus\{0\}\to \calL(X,Y)\mid j\in \mathcal{J}\}$ be a set of $X$-strongly measurable mappings for which there exists a constant $C\geq 0$ such that for each $j\in \mathcal{J}$,
$\{\abs{\xi}^{\frac{d}{r}}m_j(\xi)\mid \xi\in\Rd\}\subseteq\La(X,Y)$ is $\gamma$-bounded by $C$. Note that, since $X$ and $Y$ have finite cotype, $\gamma$-boundedness and $R$-boundedness are equivalent.
Now we claim that $\{\widetilde{T_{m_j}}\mid j\in \mathcal{J}\}\subseteq \La(\Ellp(\Rd;X),\Ellq(\Rd;Y))$ is $\gamma$-bounded as well. To prove this claim one can use the method of \cite[Theorem 3.2]{Girardi-Weis03b}. Indeed, using their notation, it follows from the Kahane-Khintchine inequalities that $\Rad(X)$ has the same type as $X$ and $\Rad(Y)$ has the same cotype as $Y$.  Therefore, given $j_1, \ldots, j_n\in \mathcal{J}$ and the corresponding $m_{j_1}, \ldots, m_{j_n}$, one can apply Theorem \ref{thm:sharpintegrability} to the multiplier $M:\R^d\setminus\{0\}\to \calL(\Rad(X),\Rad(Y))$ given as the diagonal operator with diagonal $(m_{j_1}, \ldots, m_{j_n})$. In order to check the $\gamma$-boundedness one now applies property $(\alpha)$ as in \cite[Estimate (3.2)]{Girardi-Weis03b}.
\end{remark}

\subsection{Convexity, concavity and $\Ellp$-$\Ellq$ results in lattices\label{subsec:convexconcave}}

In this section we will prove certain sharp results in $p$-convex and $q$-concave  Banach lattices.

First of all, from the proof of Theorem \ref{thm:sharpintegrability} we obtain the following result with the sharp exponents $p$ and $q$.

\begin{theorem}\label{multiplier Bessel spaces}
Let $p\in[1,2]$, $q\in[2,\infty)$, and let $r\in[1,\infty]$ be such that $\tfrac{1}{r}=\tfrac{1}{p}-\tfrac{1}{q}$. Let $X$ be a complemented subspace of a $p$-convex Banach lattice with finite cotype and $Y$ a Banach space that is continuously embedded in a $q$-concave Banach lattice. Let $m:\Rd\to\La(X,Y)$ be an $X$-strongly measurable map such that
$\{\abs{\xi}^{\frac{d}{r}}m(\xi)\mid \xi\in\Rd\setminus\{0\}\}\subseteq\La(X,Y)$ is $\gamma$-bounded. Then $T_{m}$ extends uniquely to a bounded map $\widetilde{T_{m}}\in\La(\Ellp(\Rd;X),\Ellq(\Rd;Y))$ with
\begin{align}\label{multiplier Bessel spaces norm bound}
\|\widetilde{T_{m}}\|_{\La(\Ellp(\Rd;X),\Ellq(\Rd;Y))}\leq C \gamma(\{\abs{\xi}^{\frac{d}{r}}m(\xi)\mid \xi\in\Rd\setminus\{0\}\}),
\end{align}
where $C$ is a constant depending on $X$, $Y$, $p$, $q$ and $d$.
\end{theorem}
\begin{proof}
In the case where $X$ is a $p$-convex and $Y$ is a $q$-concave Banach lattice, the embeddings in \eqref{eq:embeddingcond} can be proved in the same way as in \cite[Theorem 3.9]{Veraar13}, where the inhomogeneous case was considered. Therefore, the result in this case follows from the proof of Theorem \ref{thm:sharpintegrability}.

Now let $X_{0}$ be a $p$-convex Banach lattice with finite cotype such that $X\subseteq X_{0}$, let $P\in\La(X_{0})$ be a projection with range $X$ and let $Y_{0}$ be a $q$-concave Banach lattice with a continuous embedding $\iota:Y\hookrightarrow Y_{0}$. Let $m_{0}:\Rd\to\La(X_{0},Y_{0})$ be given by $m_{0}(\xi):=\iota\circ m(\xi)\circ P\in\La(X_{0},Y_{0})$ for $\xi\in\Rd$. It is easily checked that $\{m_{0}(\xi)\mid \xi\in\Rd\}\subseteq\La(X_{0},Y_{0})$ is $\gamma$-bounded, with
\begin{align}\label{gamma-bound bigger spaces}
\gamma(\{m_{0}(\xi)\mid \xi\in\Rd\setminus\{0\}\})\leq \|\iota\|_{\La(Y,Y_{0})}\|P\|_{\La(X_{0})} \gamma(\{\abs{\xi}^{\frac{d}{r}}m(\xi)\mid \xi\in\Rd\setminus\{0\}\}).
\end{align}
As we have shown above, there exists a constant $C\in(0,\infty)$ that depends only on $X_{0}$, $Y_{0}$, $p$, $q$ and $d$ such that $T_{m_{0}}$ extends uniquely to a bounded operator $\widetilde{T_{m_{0}}}\in\La(\Ell^p(\Rd;X_{0}),\Ell(\Rd;Y_{0}))$ with
\begin{align}\label{multiplier on bigger spaces}
\norm{T_{m_0}}_{\La(\Ell^p(\Rd;X_{0}),\Ell^q(\Rd;Y_{0}))}\leq C \gamma(\{m_{0}(\xi)\mid \xi\in\Rd\}).
\end{align}
Since $T_{m}=\widetilde{T_{m_{0}}}\!\restriction_{\Sw(\R;X)}$, the result follows from \eqref{gamma-bound bigger spaces} and \eqref{multiplier on bigger spaces}.
\end{proof}

\begin{remark}\label{dependence constant}
Note from \eqref{gamma-bound bigger spaces} and \eqref{multiplier on bigger spaces} that the constant $C$ in \eqref{multiplier Bessel spaces norm bound}
depends on $X$ and $Y$ as $C=\norm{P}_{\La(X_{0})}\norm{\iota}_{\La(Y,Y_{0})}C_{1}$, where $P\in\La(X_{0})$ is a projection with range $X$ on a $p$-convex Banach lattice $X_{0}$ with finite cotype, $\iota\in\La(Y,Y_{0})$ is a continuous embedding of $Y$ in a $q$-concave Banach lattice $Y_{0}$ and $C_{1}$ is a constant that depends only on $X_{0}$, $Y_{0}$, $p$, $q$ and $d$.
\end{remark}

\begin{remark}
By using Theorems \ref{thm:sharpintegrability} and \ref{multiplier Bessel spaces} and by multiplying in the Fourier domain by appropriate powers of $|\xi|$, versions of these theorems for multipliers from $\dot{\Hr}^{\alpha}_p(\R^d;X)$ to $\dot{\Hr}^{\beta}_q(\R^d;Y)$ can be derived. Similar results can be derived for the inhomogeneous spaces as well.
\end{remark}

So far, in all our results about $(\Ellp,\Ellq)$-multipliers the indices $p$ and $q$ have been restricted to the range $p\leq 2\leq q$, which is necessary when considering general multipliers (see \eqref{weakLrcondition}). However, we have also seen in Example \ref{ex:HLS} that for the scalar multiplier $m(\xi)=\abs{\xi}^{-s}$ such a restriction is not necessary, as follows from Proposition \ref{prop:positivekern} since the kernel associated with $m$ is positive. We now show that also for operator-valued multipliers with positive kernels on $p$-convex and $q$-concave Banach lattices, the restriction $p\leq 2\leq q$ is not necessary and moreover $\gamma$-boundedness can be avoided. First we state the result for multipliers between Bessel spaces.

\begin{theorem}\label{thm:positive kernel multipliers}
Let $p,q\in[1,\infty)$ with $p\leq q$, and let $r\in(1,\infty]$ be such that $\tfrac{1}{r}=\tfrac{1}{p}-\tfrac{1}{q}$. Let $X$ be a $p$-convex Banach lattice with finite cotype, and let $Y$ be a $q$-concave Banach lattice. Suppose that $K:\Rd\to\La(X,Y)$ is such that $K(\cdot)x\in\Ell^{1}(\Rd;Y)$ for all $x\in X$, $K(s)$ is a positive operator for all $s\in\Rd$, and $m:\R^d\to \calL(X,Y)$ is such that $\F(Kx)=mx$ for all $x\in X$. Then $T_{m}\in\La(\dot{\Hr}_{p}^{d/r}(\Rd;X),\Ellq(\Rd;Y))$ and
\begin{align}\label{eq:Hrspacepositive}
\|T_{m}\|_{\La(\dot{\Hr}_{p}^{d/r}(\Rd;X),\Ellq(\Rd;Y))}\leq C  \|m(0)\|_{\La(X,Y)}\leq C\sup_{\xi\in\Rd\setminus\{0\}} \|m(\xi)\|_{\La(X,Y)}
\end{align}
for some $C\geq 0$ independent of $K$.
\end{theorem}

By further approximation arguments one can often avoid the assumptions that $K(\cdot)x\in\Ell^{1}(\Rd;Y)$ for all $x\in X$. It follows from \cite{Rozendaal-Veraar17Stability} that the bound in Theorem \ref{thm:positive kernel multipliers} is optimal in a certain sense.

\begin{proof}
The second estimate in \eqref{eq:Hrspacepositive} follows from the continuity of $mx = \F(Kx)$. Since $\dot{\Sw}(\Rd)\otimes X$ is dense in $\dot{\Hr}^{d/r}(\Rd;X)$, for the first estimate in \eqref{eq:Hrspacepositive} it suffices to fix an $f\in \dot{\Sw}(\Rd)\otimes X$ and to show that $T_{m}(f)\in\Ellq(\Rd;X)$ with
\begin{align}\label{estimate single f}
\|T_{m}(f)\|_{\Ellq(\Rd;Y)}\leq C\|m(0)\|\, \|f\|_{\dot{\Hr}^{d/r}_{p}(\Rd;X)}.
\end{align}
Since $X$ has finite cotype, it does not contain a copy of $\ce_{0}$. Hence, by \cite[Theorem 1.a.5 and Proposition 1.a.7]{Lindenstrauss-Tzafriri79}, $X$ is order continuous. Moreover, the range of $f$ is contained in a separable subspace $X_{0}$ of $X$. By \cite[Proposition 1.a.9]{Lindenstrauss-Tzafriri79}, $X_{0}$ has a weak order unit. Now \cite[Theorem 1.b.14]{Lindenstrauss-Tzafriri79} implies that $X_{0}$ is order isometric to a Banach function space. Similarly, $Y$ is order continuous, and the range of $T_{m}(f)$ is contained in a separable subspace $Y_{0}$ which is order isometric to a Banach function space. So henceforth we may assume without loss of generality that $X$ and $Y$ are Banach function spaces.

It follows by approximation
from Lemma \ref{lem:relation different p-norms} that
\begin{align*}
\|K\ast f\|_{\Ellq(\Rd;Y)}&\leq C_{1}\|K\ast f\|_{Y(\Ellq(\Rd))}=C\Big\|\Big(\int_{\Rd}\abs{K\ast f(t)}^{q}\,\ud t\Big)^{1/q}\Big\|_{Y}\\
&=C\Big\|\Big(\int_{\Rd}\Big|\int_{\Rd}K(s)f(t-s)\,\ud s\Big|^{q}\,\ud t\Big)^{1/q}\Big\|_{Y}\\
&\leq C\Big\|\int_{\Rd}\Big(\int_{\Rd}\abs{K(s)f(t-s)}^{q}\,\ud t\Big)^{1/q}\,\ud s\Big\|_{Y}
\end{align*}
for some constant $C\geq 0$, where we used Minkowski's integral inequality in the final step. Lemma \ref{lem:positive operator through integral}, applied to the positive operator $K(s)\in\La(X,Y)$ and the function $f(\cdot-s)\in\Ellp(\Rd)\otimes X$ for each $s\in\Rd$, yields
\begin{align*}
\int_{\Rd}\Big(\int_{\Rd}\abs{K(s)f(t-s)}^{q}\,\ud t\Big)^{1/q}\,\ud s&\leq \int_{\Rd}K(s)\Big(\int_{\Rd}\abs{f(t-s)}^{q}\,\ud t\Big)^{1/q}\,\ud s\\
&=\int_{\Rd}K(s)\Big(\int_{\Rd}\abs{f(t)}^{q}\,\ud t\Big)^{1/q}\,\ud s=m(0)x_{0},
\end{align*}
where $x_{0}:=\Big(\int_{\Rd}\abs{f(t)}^{q}\,\ud t\Big)^{1/q}\in X$. Therefore,
\begin{align*}
\Big\|\int_{\Rd}\Big(\int_{\Rd}\abs{K(s)f(t-s)}^{q}\,\ud t\Big)^{1/q}\,\ud s\Big\|_{Y}&\leq \|m(0)\|\, \Big\|\Big(\int_{\Rd}\abs{f(t)}^{q}\,\ud t\Big)^{1/q}\Big\|_{X}.
\end{align*}
The Sobolev embedding $\dot{\Hr}^{d/r}_{p}(\Rd)\hookrightarrow \Ellq(\Rd)$ yields
\begin{align*}
\Big\|\Big(\int_{\Rd}\abs{f(t)}^{q}\,\ud t\Big)^{1/q}\Big\|_{X}=\|\|f(\cdot)\|_{\Ellq(\Rd)}\|_{X}\leq  C\|\|f(\cdot)\|_{\dot{\Hr}^{d/r}_{p}(\Rd)}\|_{X}.
\end{align*}
Finally, Lemma \ref{lem:relation different p-norms} yields that, for $n(\xi):=\abs{\xi}^{d/r}\I_{X}\in\La(X)$,
\begin{align*}
\|\|f(\cdot)\|_{\dot{\Hr}^{d/r}_{p}(\Rd)}\|_{X}&=\|\|T_{n}(f)(\cdot)\|_{\Ellp(\Rd)}\|_{X}\leq C\|T_{n}(f)\|_{\Ellp(\Rd;X)}=C\|f\|_{\dot{\Hr}^{d/r}_{p}(\Rd;X)}.
\end{align*}
Combining all these estimates yields \eqref{estimate single f} and concludes the proof.
\end{proof}

In terms of $\Ellp$-$\Ellq$-multipliers we obtain the following result. Note that below we require that the kernel associated with the multiplicative perturbation $\abs{\xi}^{d/r}m(\xi)$ of $m$ is positive, unlike in Proposition \ref{prop:positivekern} where this positivity was required of the kernel associated with $m$.

\begin{corollary}\label{cor:positive kernel Lp-Lq multipliers}
Let $p,q\in[1,\infty)$ with $p\leq q$, and let $r\in(1,\infty]$ be such that $\tfrac{1}{r}=\tfrac{1}{p}-\tfrac{1}{q}$. Let $X$ be a $p$-convex Banach lattice with finite cotype, and let $Y$ be a $q$-concave Banach lattice. Suppose that $K:\Rd\to\La(X,Y)$ is such that $K(\cdot)x\in\Ell^{1}(\Rd;Y)$ for all $x\in X$, $K(s)$ is a positive operator for all $s\in\Rd$,
and $m:\R^d\setminus\{0\}\to \calL(X,Y)$ is such that $\F(Kx)(\cdot)=\abs{\cdot}^{d/r}m(\cdot)x$ for all $x\in X$.
Then $T_{m}$ extends uniquely to a bounded map $\widetilde{T}_{m}\in\La(\Ellp(\Rd;X),\Ellq(\Rd;Y))$ with
\begin{align*}
\|\widetilde{T}_{m}\|_{\La(\Ellp(\Rd;X),\Ellq(\Rd;Y))}\leq C\sup_{\xi\in\Rd\setminus\{0\}}\abs{\xi}^{d/r}\|m(\xi)\|_{\La(X,Y)}
\end{align*}
for some $C\geq 0$ independent of $m$.
\end{corollary}
\begin{proof}
First note that $m$ is of moderate growth at infinity, where we use that $r>1$. Hence $T_{m}:\Sw(\Rd)\otimes X\to\Sw'(\Rd;Y)$ is well-defined. Now the result follows by applying Theorem \ref{thm:positive kernel multipliers} to the symbol $\xi\mapsto\abs{\xi}^{d/r}m(\xi)\in\La(X,Y)$, since $f\mapsto T_{\abs{\xi}^{-d/r}}(f)$ is an isometric isomorphism $\Ellp(\Rd;X)\to\dot{\Hr}^{d/r}_{p}(\Rd;X)$ and $T_{m}(f)=T_{\abs{\xi}^{d/r}m(\xi)}(T_{\abs{\xi}^{-d/r}}(f))$ for $f\in\dot{\Sw}(\Rd;X)$.
\end{proof}

\subsection{Converse results and comparison}

In the next result we show that in certain situations the type $p$ of $X$ (or cotype $q$ of $Y$) is necessary in Theorems \ref{thm:sharpintegrabilityintro}, \ref{thm:sharpintegrability} and \ref{multiplier Bessel spaces}.
The technique is a variation of the argument of Proposition \ref{prop:converseFouriertype3} and in particular Lemma \ref{lem:estimate for sharpness}.

\begin{lemma}\label{lem:estimate for sharpness-type}
Let $X$ be a Banach space with cotype $2$ and let $Y$ be a Banach space with type $2$. Let $p\in (1, 2]$, $q\in [2, \infty)$ and $r\in(1,\infty]$ be such that $\frac1r = \frac1p - \frac1q$. Assume that for all strongly measurable $m:\R^d\to \calL(X,Y)$ for which $\{\abs{\xi}^{\frac{d}{r}}m(\xi)\mid \xi\in\Rd\}$ is $\gamma$-bounded, the operator $T_{m}:\Ellp(\Rd;X)\to \Ellq(\Rd;Y)$ is bounded. Then
\begin{equation}\label{eq:toprovebiltype}
\int_{\R^d} |\xi|^{-\frac{d}{r}} \|\widehat{f}(\xi)\|_X \|\widehat{g}(\xi)\|_{Y^*} \,\ud \xi \leq C\|f\|_{\Ellp(\Rd;X)} \|g\|_{\Ell^{q'}(\Rd;Y^*)}
\end{equation}
for some $C\geq0$ and all $f\in\Sw(\Rd;X)$ and $g\in\Sw(\Rd;Y^{*})$.
\end{lemma}
At first glance it might seem surprising that we use that $X$ has cotype $2$ and $Y$ has type $2$. This is to be able to handle the $\gamma$-bound of $\{|\xi|^{\frac{d}{r}} m(\xi)\mid \xi\in \R^d\}$ in a simple way.
\begin{proof}
Since $X$ has cotype $2$ and $Y$ has type $2$, the $\gamma$-boundedness and uniform boundedness of $\{|\xi|^{\frac{d}{r}} m(\xi)\mid\xi\in \R^d\}$ are equivalent. Therefore, by the closed graph theorem there is a constant $C$ such that for all $f\in \Ellp(\Rd;X)$ and $g\in \Ell^{q'}(\Rd;Y^*)$
\[
|\lb T_m f, g\rb| \leq C \sup \{|\xi|^{\frac{d}{r}} \|m(\xi)\|_{\La(X,Y)}\mid \xi\in \R^d\} \|f\|_{\Ellp(\Rd;X)} \|g\|_{\Ell^{q'}(\Rd;Y^*)}.
\]
Hence, letting $M(\xi) := |\xi|^{\frac{d}{r}} m(\xi)$, we see that for all $f\in \Sw(\R^d;X)$ and $g\in \Sw(\R^d;Y^*)$,
\[
\Big|\int_{\R^d} \lb M(\xi) |\xi|^{-\frac{d}{r}} \wh f(\xi),\wh g(\xi)\rb \,\ud \xi\Big| \leq C \sup \{\|M(\xi)\|\mid \xi\in \R^d\} \|f\|_{\Ellp(\Rd;X)} \|g\|_{\Ell^{q'}(\Rd;Y^*)}.
\]
Taking the supremum over all strongly measurable $M$ which are uniformly bounded by $1$, a similar approximation argument as in Lemma \ref{lem:estimate for sharpness} yields the desired result.
\end{proof}

\begin{proposition}\label{prop:conversetype}
Let $X$ and $Y$ be Banach spaces. Let $p\in (1, 2]$, $q\in [2, \infty)$ and $r\in (1, \infty]$ be such that $\frac1r = \frac1p - \frac1q$. Assume that for all $X$-strongly measurable $m:\R^d\to \calL(X,Y)$ such that $\{\abs{\xi}^{\frac{d}{r}}m(\xi)\mid \xi\in\Rd\}$ is $\gamma$-bounded, the operator $T_{m}:\Ellp(\Rd;X)\to \Ellq(\Rd;Y)$ is bounded. Then the following assertions hold:
\begin{enumerate}[$(1)$]
\item\label{item:conversetype1} If $X$ has cotype $2$, $Y = \C$,  and $q=2$, then $X$ has type $p$.
\item\label{item:conversetype2} If $Y$ has type $2$, $X = \C$, and $p=2$, then $Y$ has cotype $q$.
\item\label{item:conversetype3} If $Y = X^*$ has type $2$, and $q = p'$, then $X$ has type $p$.
\end{enumerate}
\end{proposition}
\begin{proof}
First consider \eqref{item:conversetype1}. From \eqref{eq:toprovebiltype} we find that for all $f\in \Sw(\R^d;X)$ and $g\in \Sw(\R^d)$,
\[
\int_{\R^d} |\xi|^{-\frac{d}{r}} \|\widehat{f}(\xi)\|_X |\widehat{g}(\xi)| \,\ud \xi \leq C\|f\|_{\Ellp(\Rd;X)} \|\wh g\|_{\Ell^{2}(\Rd)}.\]
Taking the supremum over all $g$ with $\|g\|_{\Ell^2(\R^d)} = \|\wh g\|_{\Ell^2(\R^d)} = 1$, we obtain
\begin{align}\label{eq:Pittq=2}
\|\xi\mapsto |\xi|^{-\frac{d}{r}} \widehat{f}(\xi)\|_{\Ell^2(\R^d;X)}  \leq C\|f\|_{\Ellp(\Rd;X)}.
\end{align}
By an approximation argument this estimate extends to all $f\in \Ellp(\Rd;X)$. In particular, let $f(t) := \sum_{|k|\leq n} \one_{[-\frac{1}{2},\frac{1}{2})^d}(t+k) x_k$ for $n\in\N$, $x_{1},\ldots, x_{n}\in X$ and $t\in\Rd$. Then
\[
\|\wh f(\xi)\| = \zeta(\xi) \Big\| \sum_{|k|\leq n} e_{k}(\xi)x_k\Big\|,
\]
where $\zeta (\xi) := \prod_{j=1}^d \frac{\abs{\sin(\pi\xi_j)}}{\pi |\xi_j|}$ and $e_k(\xi) := e^{2\pi i k\cdot \xi}$ for $\xi\in\Rd$. Since $\zeta(\xi) |\xi|^{-d/r}\geq c_d$ for some $c_{d}>0$ and all $\xi\in [\frac{1}{2},\frac{1}{2}]^d$, it follows from \eqref{eq:Pittq=2} that
\[
\Big\| \sum_{|k|\leq n} e_k x_k\Big\|_{\Elltwo([-\frac{1}{2},\frac{1}{2}]^d;X)} \leq C c_d^{-1} \Big(\sum_{|k|\leq n} \|x_k\|^p\Big)^{\frac1p}.
\]
Let $(\gamma_k)_{|k|\leq n}$ be a Gaussian sequence.
Replacing $x_k$ by $\gamma_k x_k$, and taking $\Elltwo(\Omega)$-norms, we find that
\[
\Big\| \sum_{|k|\leq n} \gamma_k x_k\Big\|_{\Elltwo(\Omega;X)} \leq C c_d^{-1} \Big(\sum_{|k|\leq n} \|x_k\|^p\Big)^{\frac1p}.
\]
Here we used the fact that for each $t\in [-\frac{1}{2},\frac{1}{2}]^d$, $(\gamma_k e_k(t))_{|k|\leq n}$ is identically distributed as $(\gamma_k)_{|k|\leq n}$.
This implies that $X$ has type $p$.

Case \eqref{item:conversetype2} can be proved in a similar way by reversing the roles of $f$ and $g$. Indeed, this gives that $Y^*$ has type $q'$ and hence $Y$ has cotype $q$.

In case \eqref{item:conversetype3} we let $f = g\in \Sw(\R^d;X)$ in \eqref{eq:toprovebiltype} and argue as below \eqref{eq:Pittq=2}. Here we use that $X\subseteq X^{**}$ has cotype $2$ (see \cite[Proposition 11.10]{DiJaTo95}).
\end{proof}

If $X = \C$, then \eqref{eq:Pittq=2} is a special case of Pitt's inequality (see \cite{Beckner08} and \cite{Benedetto-Heinig03}):
\begin{align}\label{eq:Pitt}
\|\xi\mapsto |\xi|^{-\alpha} \widehat{f}(\xi)\|_{\Ell^q(\R^d;X)}  \leq C\|s\mapsto |s|^{\beta} f(s)\|_{\Ellp(\Rd;X)},
\end{align}
where $1<p\leq q<\infty$, $0\leq \alpha<\frac{d}{q}$, $0\leq \beta<\frac{d}{p'}$ and $\frac{d}{p} + \frac{d}{q} +\beta-\alpha = d$.

Note that Theorem \ref{thm:sharpintegrability} and the proof of Proposition \ref{prop:conversetype} show that \eqref{eq:Pitt} holds if $X$ has type $p_{0}>p$ and cotype $2$. Moreover, by the proof above one sees that Pitt's inequality with $\beta = 0$ and $q=2$ implies that $X$ has type $p$ and $X^*$ has type $p$. Moreover, in the case $\alpha = \beta = 0$ and $q = p'$, Pitt's inequality is equivalent to $X$ having Fourier type $p$. It seems that a vector-valued analogue of Pitt's inequality has never been studied in detail. This leads to the following natural open problem:
\begin{problem}
Characterize those Banach spaces $X$ for which Pitt's inequality \eqref{eq:Pitt} holds.
\end{problem}
For $p$-convex and $q$-concave Banach lattices, \eqref{eq:Pitt} can be proved by reducing to the scalar case using the technique of \cite[Proposition 2.2]{GTK96}.

Next we show that a $\gamma$-boundedness assumption cannot be avoided in general. In the case where $p=q$ such a result is due Cl\'ement and Pr\"uss (see \cite[Chapter 5]{HyNeVeWe16}).
In Proposition \ref{Lp-Lq multipliers Fourier type} and Theorem \ref{Lp-Lq multipliers Fourier type2} we have seen that $\gamma$-boundedness is not needed for certain $\Ell^p$-$\Ell^q$-multiplier theorems. In the following result we derive the necessity of the $\gamma$-boundedness of $\{m(\xi)\mid \xi\in \R^d\}$ under special conditions on $m$.

\begin{proposition}\label{prop:gammabddconverse}
Assume $1<p\leq q<\infty$ and let $\frac{1}{r} = \frac1p - \frac1q$. Assume $m:\R^d\to \calL(X,Y)$ is such that there is a constant $a>0$ such that $m$ takes the constant value $m_{k}$ on each of the cubes $Q_{a,k} = a([0,1]^d +k)$ with $k\in \Z^d$. If $T_m:\Ell^p(\R^d;X)\to \Ell^q(\R^d;X)$ is bounded, then
\[
\gamma(\{m_k\mid k\in \Z\})\leq  R(\{m_k\mid k\in \Z\})\leq C_{d,p,q'} a^{-d/r} \|T_m\|
\]
for some $C_{d,p,q'}\geq 0$.
\end{proposition}
In Example \ref{example sharpness R-bounds} we will provide an example where even the $\gamma$-boundedness of $\{|\xi|^{d/r} m(\xi)\mid \xi\in \R^d\}$ is necessary. However, in general such a result does not hold (see Remark \ref{rem:conversegammabdd}).
\begin{proof}
From Proposition \ref{prop:transference} and Remark \ref{rem:Torusp=q} we obtain that
\begin{equation}\label{eq:multiplierTorus}
\Big\|\sum_{|k|\leq n}  e_k m_k x_k\Big\|_{\Ell^{p}(\T^d;Y)} \leq a^{-d/r} C_{d,p,q'} \|T_m\| \, \Big\|\sum_{|k|\leq n} e_k x_k \Big\|_{\Ell^p(\T^d;X)}.
\end{equation}
Now the $R$-boundedness follows from \cite{Arendt-Bu02}. For convenience we include a short argument below. Let $(\varepsilon_k)_{|k|\leq n}$ be a sequence of independent random variables which are uniformly distributed on $\Omega:=[0,1]^d$. Replacing $x_k$ by $\varepsilon_k x_k$ in \eqref{eq:multiplierTorus} and integrating over $\Omega$ yields that
\begin{align*}
\Big\|\sum_{|k|\leq n} \varepsilon_k m_k x_k\Big\|_{\Ell^p(\Omega;Y)}  & =
\Big\|\sum_{|k|\leq n} \varepsilon_k e_k m_k x_k\Big\|_{\Ell^p(\Omega\times \T^d;Y)}
\\ & \leq a^{-d/r} C_{d,p,q'} \|T_m\| \, \Big\|\sum_{|k|\leq n} \varepsilon_k e_k x_k \Big\|_{\Ell^p(\Omega\times\T^d;X)}
\\ & \leq a^{-d/r} C_{d,p,q'} \|T_m\| \, \Big\|\sum_{|k|\leq n} \varepsilon_k x_k \Big\|_{\Ell^p(\Omega;X)}.
\end{align*}
Here we used the fact that for each $t\in \T^d$, $(\varepsilon_k e_k(t))_{|k|\leq n}$ and $(\varepsilon_k)_{|k|\leq n}$ are identically distributed.

Finally, the estimate for the $\gamma$-bound is well-known and follows from a randomization argument.
\end{proof}

The following example, which is similar to Example \ref{ex:Fouriertype}, shows that Theorem \ref{multiplier Bessel spaces} is sharp in a certain sense. In particular, it shows that the $\gamma$-boundedness condition is necessary in certain cases.

\begin{example}\label{example sharpness R-bounds}
Let $p\in[1,2]$, and for $q\in[2,\infty)$ let $r\in(1,\infty]$ be such that $\tfrac{1}{r}=\tfrac{1}{p}-\tfrac{1}{q}$. Let
$X:=\ell^{u}$ for $u\in[1,\infty)$. Let $(e_j)_{j\in\N_{0}}\subseteq X$ be the standard basis of $X$, and for $k\in\N_{0}$ let $S_{k}\in \La(X)$ be such that $S_{k}(e_{j}):= e_{j+k}$ for $j\in\N_{0}$. Let $m:\R\to\La(\ell^{u})$ be given by
$m(\xi):=\sum_{k=1}^{\infty} c_k \ind_{(k-1,k]}(\xi)S_{k}$ for $\xi\in\R$, with $c_k := k^{-\alpha} \log(k+1)^{-2}$ for $\alpha\geq 0$ arbitrary  but fixed for the moment.

Let $v\in[2,\infty]$ be such that $\tfrac{1}{v} = \big|\tfrac{1}{u} - \tfrac12\big|$. By \eqref{Kahane contraction principle} and \cite[Theorem 3.1]{vanGaans06} we find a constant $C\geq 0$ such that
\begin{align*}
\gamma(\{\abs{\xi}^{\frac1r} m(\xi)\mid \xi\in\R\}) &\leq \gamma(\{k^{\frac1r - \alpha} \log(k+1)^{-2} S_{k}\mid k\in\N\})
\\ & \leq C \|(k^{\frac1r - \alpha} \log(k+1)^{-2} \|S_k\|_{\La(X)})_{k=1}^\infty\|_{\ell^v} \\
&\leq C\Big(\sum_{k=1}^{\infty} k^{(\frac1r - \alpha)v} \log(k+1)^{-2v}  \Big)^{\frac1v}
\end{align*}
(with the obvious modification for $v=\infty$), and the latter expression is finite if and only if $\frac1r - \alpha\leq -\frac{1}{v}$, i.e.~if and only if $\alpha\geq \tfrac{1}{p}-\tfrac{1}{q}+\tfrac{1}{v}$.

If $u\in [p, 2]$ then $X$ is a $p$-convex and $q$-concave Banach lattice for all $q\geq p$, hence by Theorem \ref{multiplier Bessel spaces} we find that with $\alpha =\tfrac{1}{p}-\tfrac{1}{q}+\tfrac{1}{u} - \tfrac{1}{2}$, $T_m:\Ellp(\R;X)\to \Ellq(\R;X)$ is bounded for all $q\geq 2$. Note that for $q=2$ and $u>p$, $m$ is more singular than in Example \ref{ex:Fouriertype}, where we used Proposition \ref{Lp-Lq multipliers Fourier type} to obtain the boundedness of $T_{m}:\Ellp(\R;X)\to\Ell^{p'}(\R;X)$ for $\alpha=\tfrac{1}{p}-\tfrac{1}{p'}> \tfrac{1}{p}+\tfrac{1}{u}-1$. In the special case where $u=p$, both results can be combined using complex interpolation to obtain that $T_m:\Ell^p(\R;X)\to \Ell^q(\R;X)$ is bounded for all $q\in [2, p']$ if $\alpha = \tfrac{2}{p}-1$.

Note also that the difference between Proposition \ref{Lp-Lq multipliers Fourier type} and Theorem \ref{multiplier Bessel spaces} is most pronounced when $p=u=1$. In this case $X=\ell^{1}$ has trivial type and trivial Fourier type, but cotype $q=2$. Hence Proposition \ref{Lp-Lq multipliers Fourier type} only yields the boundedness of $T_{m}:\Ell^{1}(\R;X)\to\Ellinfty(\R;X)$ for $\alpha\geq 1$, which can also be obtained trivially since in this case $m$ is integrable. On the other hand, Theorem \ref{multiplier Bessel spaces} yields the nontrivial statement that $T_{m}:\Ell^{1}(\R;X)\to\Ell^{2}(\R;X)$ is bounded for $\alpha\geq 1$.

Now fix $q\in[2,\infty)$ and let $u\in [2, q]$. Then, similarly, with $\alpha = 1 - \frac{1}{q} - \frac1u$ the operator $T_m:\Elltwo(\R;X)\to \Ellq(\R;X)$ is bounded. In the special case that $u = q$, combined with Example \ref{ex:Fouriertype} we find that $T_m:\Ellp(\R;X)\to \Ellq(\R;X)$ is bounded for all $p\in [q',2]$ with $\alpha = \tfrac{2}{q'}-1$.

We now show that in certain cases the condition $\alpha \geq \frac{1}{p} - \frac{1}{q} + \Big|\frac{1}{u}  - \frac12\Big|$ for the $\gamma$-boundedness of $\{\abs{\xi}^{1/r}m(\xi)\mid \xi\in\R\}$ from above is sharp in order for $T_m:\Ellp(\R;X)\to \Ellq(\R;X)$ to be bounded. First suppose that $u\in [1, 2]$. For $k\in\N$ let $\ph_{k}:\R\to\C$ be such that $\widehat{\ph_{k}}=\one_{(k-1,k]}$, and for $n\in\N$ let $f := \sum_{k=n+1}^{2n} \ph_{k} e_0$. Then, as in Example \ref{ex:Fouriertype}, we find that
\begin{align*}
\|T_{m}(f)\|_{\Ellq(\R;X)} \geq  n^{\frac1u} \abs{c_{2n}}  \|\varphi_1\|_{\Ellq(\R)} = C n^{\frac{1}{u}} n^{-\alpha} \log(n)^{-2}
\end{align*}
and $\|f\|_{\Ellp(\R;X)} \leq C_{2} n^{1-\frac{1}{p}}$ for $p>1$. Therefore, $\alpha \geq \frac{1}{u} + \frac{1}{p}-1$. This shows that for $q=2$ and $u\in [1, 2]$, the condition on $\alpha$ which guarantees $\gamma$-boundedness is necessary. In the case $u\in [2, \infty)$, a duality argument shows that
$\alpha \geq \frac{1}{u'} + \frac{1}{q'}-1 = 1-\frac{1}{u} - \frac1q$, which shows that the $\gamma$-boundedness condition is also necessary if $p=2$ and $u\in [2, \infty)$.

Recall from the last part of Example \ref{ex:Fouriertype} that if $u\in [1, \infty)$ and $\alpha = \frac{2}{p}-1$ and $T_m:\Ell^p(\R;X)\to \Ell^2(\R;X)$ is bounded, then $\frac{1}{u}\leq 1-\frac1p + \alpha = \frac{1}{p}$ and thus $u\geq p$. Similarly, if $T_m:\Ell^2(\R;X)\to \Ell^q(\R;X)$ is bounded with $\alpha = 1-\frac{2}{q}$, then $\frac{1}{u'}\leq 1-\frac1{q'} + \alpha = \frac{1}{q'}$, and thus $u\leq q$.

By considering $m_n(\xi) := \sum_{k=1}^{n} \ind_{(k-1,k]}(\xi)S_{k}$ a similar argument yields that for $X = \ell^p$ with $p\in [1, 2]$ and $\frac1r = \frac1p - \frac12$, one has
\[\|T_{m_n}\|_{\calL(\Ell^p(\R;X), \Ell^2(\R;X))} \eqsim_p \gamma(\{|\xi|^{\frac1r} m_n(\xi)\mid \xi\in \R\}).\]
In particular this shows that the $\gamma$-bound provides the right factor in certain cases.
\end{example}

In the following remark we show that one cannot prove the $\gamma$-boundedness, or even the uniform boundedness, of $\{|\xi|^{d/r} m(\xi)\mid \xi\in \R^d\}$ in general.

\begin{remark}\label{rem:conversegammabdd}
Let $m:\R^d\setminus\{0\}\to \calL(X,Y)$ be $X$-strongly measurable. If $r<\infty$, then one cannot prove
\[
\sup\{|\xi|^{\sigma} \|m(\xi)\|\mid \xi\in \R^d\setminus\{0\}\}\leq C \|T_m\|
\]
for any $\sigma\in \R$.
Indeed, $\sigma\leq 0$ is not possible for the multiplier $m(\xi) := |\xi|^{-d/r}$ which is unbounded near zero.
For $\sigma>0$, one can use the same multiplier and a translation argument to deduce a contradiction. Moreover, for any nonzero multiplier $m$ one can consider $m_h = m(\cdot-h)$ for $h\in \R^d$. Then $\|T_m\| = \|T_{m_h}\|$ and it follows that
\[
|\xi_0+h|^{\sigma} \|m(\xi_0)\| = \sup\{|\xi|^{\sigma} \|m(\xi-h)\|\mid \xi\in \R^d\setminus\{0\}\}\ \leq C \|T_{m_h}\| = C \|T_{m}\|
\]
for all $\xi_0\in \R^d$. Letting $|h|\to \infty$ yields a contradiction whenever $m(\xi_0)\neq 0$.
\end{remark}

In the next remark we compare the results obtained in this section with the ones obtained by Fourier type methods.

\begin{remark}\
\begin{enumerate}[(i)]
\item Consider the case of scalar-valued multipliers $m$. If $X = Y$ has Fourier type $p_0>p$, then Theorem \ref{Lp-Lq multipliers Fourier type2} states that $T_m\in \calL(\Ell^p(\R^d;X),\Ell^{p'}(\R^d);X))$ for all $m\in \Ell^{r,\infty}(\R^d)$, where $\frac{1}{r} = \frac{1}{p} - \frac{1}{p'}$. This class of multipliers is larger than the one obtained in Theorem \ref{thm:sharpintegrability} since
\[
\sup\{ |\xi|^{\frac{d}{r}} m(\xi)\mid \xi\in \R^d\} \leq C_d \|m\|_{\Ell^{r,\infty}(\R^d)}.
\]
On the other hand, the geometric conditions in Theorem \ref{thm:sharpintegrability} are less restrictive. Indeed, Fourier type $p_0$ implies that $X$ has type $p_0$ and cotype $p_0'$, but the converse is false.
\item An important difference between Proposition \ref{Lp-Lq multipliers Fourier type} and
Theorem \ref{Lp-Lq multipliers Fourier type2} and the results obtained in Subsections \ref{sec:Bessel1} and \ref{subsec:convexconcave} is that the former do not require any $\gamma$-boundedness condition. Of course the assumptions on type and cotype are less restrictive, and furthermore by \cite{Hytonen-Veraar09} the $\gamma$-boundedness can be avoided if $X$ has cotype $u$ and $Y$ has type $v$ and $|\cdot|^{\frac{d}{r}} m(\cdot)\in B^{\frac{d}{w}}_{w,1}(\R^d;\calL(X,Y))$ for $\frac1w = \frac{1}{u}-\frac{1}{v}$. In this case
\[
\gamma(\{ |\xi|^{\frac{d}{r}} m(\xi)\mid\xi\in \R^d\}) \leq \||\cdot|^{\frac{d}{r}} m(\cdot)\|_{\Be^{\frac{d}{w}}_{w,1}(\R^d;\calL(X,Y))}.
\]
\end{enumerate}
\end{remark}

\section{Extrapolation}\label{extrapolation}

In this section we briefly discuss an extension of the extrapolation results of H\"ormander in \cite{Hormander60}.

Let $m:\Rd\setminus\{0\}\to\La(X,Y)$ be a strongly measurable map of moderate growth at zero and infinity.
For $r\in [1, \infty)$, $\varrho\in [1, \infty)$ and $n\in \N$, consider the following variants of the Mihlin--H\"ormander condition:
\begin{enumerate}
\item[(M1)$_{r,\varrho,n}$]
There exists a constant $M_1\geq 0$ such that for all multi-indices $|\alpha|\leq n$,
\begin{align*}
R^{|\alpha| + \frac{d}{r} - \frac{d}{\varrho}} \Big(\int_{R\leq |\xi|<2R} \|\partial^{\alpha} m(\xi)x\|^{\varrho} \, d\xi\Big)^{1/\varrho}\leq M_1\|x\|\qquad (x\in X, R>0).
\end{align*}
\item[(M2)$_{r,\varrho,n}$]
There exists a constant $M_2\geq 0$ such that for all multi-indices $|\alpha|\leq n$
\begin{align*}
R^{|\alpha| + \frac{d}{r} - \frac{d}{\varrho}} \Big(\int_{R\leq |\xi|<2R} \|\partial^{\alpha} m(\xi)^*y^*\|^{\varrho} \, d\xi\Big)^{1/\varrho}\leq M_2\|y^*\|\qquad (y^*\in Y^*,  R>0).
\end{align*}
\end{enumerate}
In the case $\varrho = 2$, $r=1$, $X = Y = \R$, condition (M1)$_{r,\varrho,n}$ reduces to the classical H\"ormander condition in \cite[Theorem 2.5]{Hormander60} (see also \cite[Theorem 5.2.7]{Grafakos08}).

Now we can formulate the main result of this section. It extends \cite[Theorem 2.5]{Hormander60} to the vector-valued setting and to general exponents $p,q\in (1, \infty)$.

\begin{theorem}[Extrapolation]\label{thm:extrapolmult}
Let $p_0,q_0,r\in [1, \infty]$ with $r\neq 1$ be such that $\frac{1}{p_0} - \frac{1}{q_0} = \frac{1}{r}$. Let $m:\Rd\setminus\{0\}\to\La(X,Y)$ be a strongly measurable map of moderate growth at zero and infinity. Suppose that $T_m:\Ell^{p_0}(\R^d;X) \to \Ell^{q_0}(\R^d;Y)$ is bounded of norm $B$.
\begin{enumerate}[$(1)$]
\item Suppose that $p_0\in (1, \infty]$, $Y$ has Fourier type $\varrho\in [1, 2]$ with $\varrho\leq r$, and (M1)$_{r, \varrho, n}$ holds for $n := \lfloor \frac{d}{\varrho} - \frac{d}{r} \rfloor+1$.
Then $T_{m}\in\La(\Ellp(\Rd;X),\Ellq(\Rd;Y))$ and
\begin{align}\label{eq:toprovebddextr}
\|T_{m}\|_{\La(\Ellp(\Rd;X),\Ellq(\Rd;Y))} \leq C_{p_{0},q_{0},p,d} (M_1 + B)
\end{align}
for all $(p,q)$ such that $p\in (1, p_0]$ and $\tfrac{1}{p}-\tfrac{1}{q}=\tfrac{1}{r}$, where $C_{p_0, q_0, p, d}\sim (p-1)^{-1}$ as $p\downarrow 1$.
\item Suppose that $q_0\in (1, \infty)$, $X$ has Fourier type $\varrho\in [1, 2]$ with $\varrho\leq r$, and (M2)$_{r, \varrho, n}$ holds for $n := \lfloor \frac{d}{\varrho} - \frac{d}{r} \rfloor+1$. Then $T_{m}\in\La(\Ellp(\Rd;X),\Ellq(\Rd;Y))$ and
\begin{align}\label{eq:toprovebddextr2}
\|T_m\|_{\calL(\Ellp(\R^d;X),\Ellq(\R^d;Y))} \leq C_{p_{0},q_{0},q,d} (M_2 + B),
\end{align}
for all $(p,q)$ satisfying $q\in [q_0,\infty)$ and $\tfrac{1}{p}-\tfrac{1}{q}=\tfrac{1}{r}$, where $C_{p_0, q_0, q, d}\sim q$ as $q\uparrow \infty$.
\end{enumerate}
\end{theorem}

The proof will be presented in \cite{Rozendaal-Veraar16Besov}.
It is based on the classical argument in the case $p=q$ (see \cite[Theorem 5.2.7]{Grafakos08}). One of the other ingredients is an operator-valued analogue of \cite[Theorem 2.2]{Hormander60}.

As a consequence we obtain the following extrapolation result:

\begin{corollary}\label{cor:extrapol m uniform}
Let $p_0,q_0,r\in [1, \infty]$ with $q_0\neq1 $ and $r\neq 1$ be such that $\frac{1}{p_0} - \frac{1}{q_0} = \frac{1}{r}$.
Let $X$ and $Y$ both have Fourier type $\varrho\in [1, 2]$ $\varrho\leq r$ and let $n := \lfloor \frac{d}{\varrho} - \frac{d}{r} \rfloor+1$.
Let $m:\Rd\setminus\{0\}\to\La(X,Y)$ be such that, for all multi-indices $|\alpha|\leq n$,
\begin{align}\label{eq:mihlin}
\|\partial^{\alpha} m(\xi)\|\leq C|\xi|^{-|\alpha| - \frac{d}{r}}, \ \ \xi\in \R^d\setminus \{0\}.
\end{align}
Suppose that $T_m:\Ell^{p_0}(\Rd;X) \to \Ell^{q_0}(\Rd;Y)$ is bounded of norm $B$.
Then, for all exponents $p$ and $q$ satisfying $1<p\leq q<\infty$ and $\frac1p-\frac1q=\frac1q$, $T_m:\Ellp(\R^d;X) \to \Ellq(\R^d;Y)$ is bounded and
\begin{align*}
\|T_m\|_{\La(\Ellp(\Rd;X),\Ellq(\Rd;Y))} \leq C_{p,q,d} (B+C)
\end{align*}
for some constant $C_{p,q,d}\geq 0$.
\end{corollary}
In particular, one can always take $\varrho= 1$ and $n= \lfloor \frac{d}{r'}\rfloor+1$ in the above results.
\begin{proof}
Note that, for $\xi\in\Rd$, $x\in X$ and $y^{*}\in Y^{*}$, $\|m(\xi)x\|_{Y}\leq \|m(\xi)\|_{\La(X,Y)}\, \|x\|_{X}$ and $\|m^* y^*\|_{X^{*}}\leq \|m(\xi)\|_{\La(X,Y)} \, \|y^*\|_{Y^{*}}$, and similarly for the derivatives of $m$.
Therefore, the result follows from Theorem  \ref{thm:extrapolmult} (1) and (2). Indeed,
\begin{enumerate}[(i)]
\item $p_0, q_0\in (1, \infty)$: apply (1) and (2).
\item $p_0\in (1, \infty]$, $q_0 = \infty$: apply (1).
\item $p_0=1$, $q_0\in (1, \infty)$: apply (2).
\item $p_0=1$, $q_0 =\infty$ is not possible, since $r\neq 1$.
\item $p_0=1$, $q_0 = 1$ is not possible, since $q_0\neq 1$.
\end{enumerate}
 \end{proof}

If $p_0 = q_0=1$, then Theorem \ref{thm:extrapolmult} and Corollary \ref{cor:extrapol m uniform} are true with $\varrho=1$ (see \cite{Rozendaal-Veraar16Besov}).

Next we consider several applications of these extrapolation results.

In \cite{Lizorkin67} an $\Ellp$-$\Ellq$-Fourier multiplier result was proved assuming differentiability up to order $d$. Moreover, in \cite{SaTaTl10} an extension is discussed in the case $d=1$. We prove a similar result in the Hilbert space case in arbitrary dimensions assuming less differentiability.

\begin{example}\label{ex:scalar extrapolation}
Let $X$ and $Y$ be Hilbert spaces. First consider $r\in (2, \infty]$ and let $n:=\lfloor d(\tfrac{1}{2}-\tfrac{1}{r}) \rfloor+1$ and assume that $m:\R^d\setminus\{0\}\to \C$ is such that for all $|\alpha|\leq n$
\begin{equation}\label{eq:mihlin2}
|\partial^{\alpha} m(\xi)|\leq C|\xi|^{-|\alpha| - \frac{d}{r}}\qquad( \xi\in \R^d\setminus \{0\}).
\end{equation}
Then $T_m:\Ell^p(\R^d;X)\to \Ell^q(\R^d;X)$ is bounded for all $1<p\leq q<\infty$ such that $\tfrac{1}{p}-\tfrac{1}{q}=\tfrac{1}{r}$.
Indeed, we first prove the boundedness of $T_m$ in special cases.
If $r= \infty$, then one can take $p_0 = q_0 = 2$ and the boundedness of $T_m$ from $\Ell^2(\R^d;X)$ into $\Ell^2(\R^d;Y)$ follows from Plancherel's isometry and the uniform boundedness of $m$. If $r<\infty$, then we can find $p_{0}\in(1,2)$ and $q_{0}\in(2,\infty)$ such that $\tfrac{1}{p_{0}}-\tfrac{1}{q_{0}}=\tfrac{1}{r}$. Since $X$ and $Y$ have Fourier type $2$ the boundedness of $T_m$ from $\Ell^{p_0}(\R^d;X)$ into $\Ell^{q_0}(\R^d;Y)$ follows from Theorem \ref{Lp-Lq multipliers Fourier type2}.
Now Corollary \ref{cor:extrapol m uniform} can be applied to extrapolate the boundedness to the remaining cases.

Next let $r\in(1,2]$. Then all $p,q\in(1,\infty)$ satisfying $\tfrac{1}{p}-\tfrac{1}{q}=\tfrac{1}{r}$ are such that $p\in(1,2)$ and $q\in(2,\infty)$. Hence each $m$ satisfying \eqref{eq:mihlin2} for $\alpha=0$ yields a bounded operator $T_{m}:\Ellp(\Rd;X)\to\Ellq(\Rd;Y)$ for all such $p,q$ by Theorem \ref{Lp-Lq multipliers Fourier type2}.
\end{example}

\begin{remark}\label{rem:less derivatives}
Even in the case where $X = Y = \C$ (or $X$ and $Y$ are Hilbert spaces) the result in Corollary \ref{cor:extrapol m uniform} with $\rho=2$ was only known for $r=\infty$. The point is that we only need derivatives up to order $\lfloor d(\tfrac{1}{2} - \tfrac{1}{r})\rfloor+1$ if $r>2$, whereas the classical condition requires derivatives up to  $\lfloor d/2 \rfloor+1$. However, if $m$ would have derivatives up to order $n:=\lfloor d/2\rfloor+1$ for which \eqref{eq:mihlin} holds, then the multiplier $M(\xi) := |\xi|^{d/r} m(\xi)$ would satisfy the classical Mihlin condition:
for all $|\alpha|\leq n$
\begin{align*}
\|\partial^{\alpha} M(\xi)\|\leq C|\xi|^{-|\alpha|} \qquad(\xi\in \Rd\setminus \{0\}).
\end{align*}
Therefore, $T_{M}\in\La(\Ellp(\Rd),\Ellp(\Rd))$ for all $p\in (1, \infty)$. Consequently we find that, for any $1<p\leq q<\infty$ with $\frac{1}{p}-\frac{1}{q} = \frac1r$,
\[
\|T_m\|_{\calL(\Ellp(\Rd),\Ellq(\Rd))}\leq \|T_M\|_{\calL(\Ellp(\Rd),\Ellp(\Rd))} \|T_{\abs{\xi}^{-d/r}}\|_{\calL(\Ellp(\Rd),\Ellq(\Rd))} <\infty,
\]
where we used the Hardy-Littlewood-Sobolev inequality (see Example \ref{ex:HLS}). For $r\leq 2$ we have already observed in Example \ref{ex:scalar extrapolation} that in the Hilbertian setting no derivatives are required.

Thus in the scalar or Hilbertian setting we emphasize that the only new point is that less derivatives are required of the multiplier for $p<q$.

In the case where $X$ and $Y$ are general Banach spaces, the assertion about $T_{\abs{\xi}^{-d/r}}$ remains true. However, the boundedness of $T_M$ is not as simple to obtain and in general requires geometric conditions on $X$ (even if $m$ is scalar-valued) and an $R$-boundedness version of the Mihlin condition (see \cite{Kunstmann-Weis04}).
\end{remark}

Another application of Corollary \ref{cor:extrapol m uniform} is that we can extrapolate the result of Theorem \ref{thm:sharpintegrability} to other values of $p$ and $q$. A similar result holds for Theorem \ref{multiplier Bessel spaces}.

\begin{corollary}\label{cor:ExtraAppl}
Let $X$ be a Banach space with type $p_0\in(1,2]$ and $Y$ a Banach space with cotype $q_0\in[2,\infty)$, and let $p_1\in (1, p_0)$ and $q_1\in (q_0, \infty)$, $r\in[1,\infty]$ be such that $\tfrac{1}{r}=\tfrac{1}{p_1}-\tfrac{1}{q_1}$. Let $m:\Rd\setminus\{0\}\to\La(X,Y)$ be such that
$\{\abs{\xi}^{\frac{d}{r}}m(\xi)\mid \xi\in\Rd\setminus\{0\}\}\subseteq\La(X,Y)$ is $\gamma$-bounded.

Assume that $X$ and $Y$ both have Fourier type $\varrho\in [1, 2]$ with $\varrho\leq r$ and let $ n:= \lfloor d(\tfrac{1}{\varrho}-\tfrac{1}{r})\rfloor+1$.
Assume for all multi-indices $|\alpha|\leq n$
\begin{equation}\label{eq:mihlinIk}
\|\partial^{\alpha} m(\xi)\|\leq C|\xi|^{-|\alpha| - \frac{d}{r}}\qquad (\xi \in \Rd\setminus\{0\}).
\end{equation}
Then $T_{m}$ extends uniquely to a bounded map $\widetilde{T_{m}}\in\La(\Ellp(\Rd;X),\Ellq(\Rd;Y))$ for all $1<p\leq q<\infty$ satisfying $\frac1p-\frac1q=\frac{1}{r}$.
\end{corollary}

\begin{proof}
The case where $p=p_1$ and $q=q_1$ follows from Theorem \ref{thm:sharpintegrability}. The result for the remaining values of $p$ and $q$ follows from Corollary \ref{cor:extrapol m uniform}.
\end{proof}

\bibliographystyle{plain}
\bibliography{Bibliografie}

\end{document}